\documentclass[11pt]{article}
\usepackage{amsmath,amssymb,amsthm}
\usepackage[utf8]{inputenc}
\usepackage{mathtools}
\usepackage{graphicx}

\newtheorem{thm}{Theorem}
\newtheorem{lem}{Lemma}
\newtheorem{cor}{Corollary}
\newtheorem{prop}{Proposition}
\newtheorem{rem}{Remark}
\newtheorem{defi}{Definition}
\newtheorem*{defi*}{Definitions}
\newtheorem{question}{Question}

\newcommand{\ds}{\displaystyle}
\newcommand{\eps}{\varepsilon}
\newcommand{\ep}{\epsilon}
\newcommand{\R}{\mathbb{R}}

\newcommand{\N}{\mathbb{N}}
\newcommand{\Z}{\mathbb{Z}}
\newcommand{\T}{\mathbb{T}}
\newcommand{\mb}{\mathversion{bold}}
\newcommand{\mn}{\mathversion{normal}}
\def\rest{\hskip 1pt{\hbox to 10.8pt{\hfill
\vrule height 7pt width 0.4pt depth 0pt\hbox{\vrule height 0.4pt
width 7.6pt depth 0pt}\hfill}}}
\def\evalu{\hskip 1pt{\hbox to 2pt{\hfill \vrule height -6pt width 0.4pt depth
0pt}}}
\def\barint{\mathop{\vrule width 6pt height 3 pt depth -2.5pt \kern -8.8pt
\intop}}

\title{On the motion of a curve by its binormal curvature}
\author{
{\sc Robert L. Jerrard}
 \ \& \ {\sc Didier Smets}
}

\date{}
\begin{document}

\maketitle

\begin{flushright}
``{\it I confess, I am skeptical about the stability of many\\
of the motions which you appear to contemplate.}''

\smallskip

Stokes, letter to Kelvin, 1873. 
\end{flushright}

\begin{abstract} We propose a weak formulation for the binormal curvature flow of curves in $\R^3.$  This formulation
 is sufficiently broad  to consider integral currents as initial data, and sufficiently strong for the 
weak-strong uniqueness property to hold, as long as self-intersections do not occur. We also prove a global 
existence theorem in that framework.
\end{abstract}

\section{Introduction}                                              %
The binormal curvature flow equation for a smooth family $(\gamma_t)_{t\in I}$ of curves in $\R^3$ 
is traditionally written in terms of an arc-length parametrization $\gamma\: : \: I \times \R \to \R^3$ 
by 
\begin{equation}\label{eq:strongbfbis}
\partial_t \gamma = \partial_s \gamma \times \partial_{ss}\gamma 
\end{equation}
where $t\in I$ is the time variable, $s\in \R$ is the arc-length parameter, and $\times$ denotes the 
vector product in $\R^3.$ The arc-length parametrization condition 
\begin{equation}\label{eq:arc}
|\partial_s \gamma(t,s)|^2 = 1
\end{equation}
is indeed compatible with equation \eqref{eq:strongbfbis}, since 
$$
\partial_t \bigl( |\partial_s \gamma|^2\bigr) = 2 \partial_s\gamma \cdot
\partial_{st}\gamma = 2 \partial_s\gamma \cdot \bigl( \partial_s\gamma \times
\partial_{sss} \gamma\bigr) =0
$$  
whenever \eqref{eq:strongbfbis} is satisfied, at least for sufficiently smooth solutions. In particular,
 closed curves evolved by the binormal curvature flow equation \eqref{eq:strongbfbis} all have constant 
length. In more geometric terms, equation \eqref{eq:strongbfbis} takes its name from its equivalent form   
$$
\partial_t \gamma = \kappa b
$$
where $\kappa$ and $b$ are the curvature function and the binormal
vector field along $\gamma_t$ respectively.  

\medskip

It seems that equation \eqref{eq:strongbfbis} first appeared in the 1906 Ph.D. thesis of L.S. Da Rios 
\cite{DaR},  whose work was promoted in a series of lectures in 1931 in Paris by its advisor 
T. Levi-Civita \cite{LeC}. The problem considered by Da Rios and Levi-Civita goes back to the celebrated
1858 paper of H. Helmholtz \cite{Hel1} on the motion of a three dimensional incompressible fluid in
rotation.  Special attention was paid in the second part of \cite{Hel1} to configurations called 
``unendlich kleine Querschnitts'', and translated in \cite{Hel2} by vortex-filaments of indefinitely small
cross-section: in such configurations, 
 the vorticity field $\omega := {\rm
curl}(v) $ associated to the velocity field $v$ of the fluid at a given time $t$ is concentrated along a 
closed oriented curve $\gamma_t$, parallel to it and vanishing rapidly away from it, so that 
$$
\int_{\R^3} X(x)\cdot \omega(x,t)\, dx \simeq \int_{\gamma_t} X \cdot \tau_{\gamma_t} d\mathcal{H}^1
$$
in some appropriate sense for any vector field $X\in \mathcal{D}(\R^3,\R^3).$ Helmholtz, like everybody since, 
failed to rigorously answer the question of the persistence in time of such vortex-filaments under the Euler 
flow
$$
\partial_t \omega + v\cdot \nabla \omega = \omega\cdot \nabla \omega.
$$
Nevertheless, he obtained a number of important contributions in that direction, as well as suggestive evidences, 
which conducted him to study the question of the corresponding asymptotic motion law for the 
underlying curves $\gamma_t$ in case of positive answer to the previous question. Because of mathematical 
obstacles related to the singularity of the Biot-Savart kernel involved in the reconstruction of $v$ 
from $\omega$ when considering such vorticity measures, Helmholtz essentially restricted his mathematical 
study to the case of straight or circular vortex-filaments, or combinations of those. Pursuing Helmholtz work, 
Lord Kelvin announced in 1867 \cite{Kel1} and published in 1880 \cite{Kel2} the first result on linear 
stability of circular vortex-filaments.  The latter, also called vortex rings, correspond in the
asymptotic of infinitely small cross-section to the traveling wave solutions of equation \eqref{eq:strongbfbis} 
given by 
$$
\gamma(t,s) = \gamma_{r,\vec e}(s) + \frac{t}{r}\vec e, 
$$
where $\gamma_{r,\vec e}$ is an arc-length parametrization of a circle of radius $r$ in a plane perpendicular to
the unitary vector $\vec e \in \R^3.$ Kelvin carefully described the neutral modes involved in small perturbations of
such configurations, and which are referred today as Kelvin waves. J.J. Thomson 1883 treatise \cite{Tho} 
and H. Poincar\'e 1893 lectures notes \cite{Poi} are also important sources regarding the state of the art for 
vortex-filaments motion in incompressible fluids by the end of the nineteenth century. As already mentioned, it is
only in 1906 with a careful use of potential theory that Da Rios formally obtained the speculated general motion law
\eqref{eq:strongbfbis}.  

\medskip

Let aside the fact that it has never been rigorously derived from the Euler equations, 
and even though is is globally well-posed for initial data consisting of smooth closed curves, 
formulation \eqref{eq:strongbfbis} for binormal curvature flows has at least two limitations which we would like 
to address. 

\smallskip
First, by essence this formulation is tailored for parametrized curves. In particular, 
and since it involves derivatives with respect to the parameters only, it is necessarily insensitive to 
self-intersections\footnote{By self-intersection of $\gamma_t$ 
we mean failure of injectivity of the map $\gamma(t,\cdot).$} in the curves $\gamma_t.$ This property is surely 
unsatisfactory if one believes that such flows arise as limits from three dimensional fluid dynamics. Instead, 
it would be desirable for a formulation to be able to detect such self-intersections, as well as possible
collisions between elements of disconnected vortex filaments and changes of topology. 
     
Second, there are presumably important configurations of curves which are too singular to be considered under
formulation \eqref{eq:strongbfbis}. Indeed, invoking distributional derivatives one can give a meaning to 
equation \eqref{eq:strongbfbis} in a variety of spaces, but those spaces just fail to include the case of 
curves which are barely Lipschitz. On the other hand, in numerical simulations of
the Euler equation or the Gross-Pitaevskii equation for quantum fluids, it is observed (see e.g.  \cite{KiTa} 
and \cite{KoLe}) that vortex-filaments often tend to recombine by exchanging strands in cases of collisions or 
self-intersections. Those recombinations, when the intersections are transverse, inevitably create 
discontinuities of the tangent vector (see Figure \ref{fig:1} below).   

\begin{figure}[!h]
\begin{center}
\includegraphics{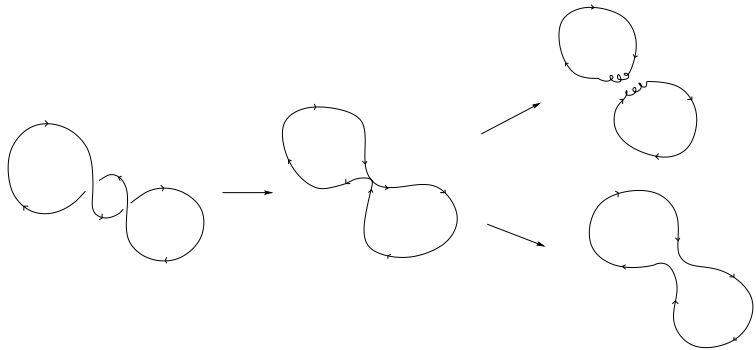}
\caption{Non unique evolution through strands recombination and singularity formations.}
\label{fig:1}
\end{center}
\end{figure}

Our starting point in trying to address these two important limitations is the following identity for smooth
solutions of \eqref{eq:strongbfbis}, which was remarked by the first author in \cite{Je1} in a more general 
context.  

\begin{lem}[\cite{Je1}]\label{lem:weakform}
If $\gamma$ is a smooth solution of \eqref{eq:strongbfbis} on $I\times \T^1$, where $I\subset \R$ is some open
interval and $\T^1 = \R/\ell\mathbb{Z}$ for some $\ell>0,$ then for every vector
field $X\in \mathcal{D}(\R^3,\R^3)$ and every $t\in I$
\begin{equation}
\label{eq:weakbf1}
\frac{d}{dt} \int_{\gamma_t} X \cdot \tau_t \, d\mathcal{H}^1 
= -\int_{\gamma_t} D\left( {\rm curl} X\right): \left( \tau_t \otimes
\tau_t\right)
\, d\mathcal{H}^1,
\end{equation} 
where $\gamma_t\equiv \gamma(t,\cdot)$ and $\tau_t$ is the oriented tangent vector along 
$\gamma_t$. 
\end{lem}

Notice that for fixed time, both sides of \eqref{eq:weakbf1} involve, in terms of $\gamma$, 
only the tangent vectors $\tau_t,$ and therefore first order derivatives with respect to the
arc-length. This suggests to enlarge the definition of binormal curvature flows through 
an extension of formula \eqref{eq:weakbf1} to one dimensional objects that
have well defined tagent spaces, at least in a measure theoretic sense.
A tentative definition based entirely on integral currents of H. Federer and W. H. Fleming \cite{FeFl}
was first proposed in \cite{Je1}; an existence theory in that framework is still missing.    
The main difficulty in dealing with \eqref{eq:weakbf1} in the framework of currents is that the 
right-hand side doesn't have good continuity properties for the usual topologies associated
to currents, because of the presence of quadratic terms in the tangent vectors. Instead, 
such quantities seem more appropriate to be dealt with using the general framework of 
Young measures, and more specifically varifolds of F. J. Almgren \cite{Alm} and W. K. Allard \cite{All}.  
On the other hand, the left-hand side of \eqref{eq:weakbf1} is more appropriate to currents
than varifolds, in particular because the latter do not have an orientation. The strategy which we adopt
here below tries in a sense to reconcile these two features, building both on integral currents
and on a notion of oriented varifolds which can be viewed as the non-parametric version of what L.C. Young 
\cite{You} and E. J. McShane \cite{McS} called generalized curves.

 
\noindent
{\bf Integral currents.} H. Federer and W. H. Fleming introduced integral currents of arbitrary dimension 
in \cite{FeFl}. One dimensional currents have a simple characterization which we adopt as a definition (see
\cite{Fed} 4.2.25).

A simple closed oriented curve in $\R^3$ is a vector valued distribution 
$T\in \mathcal{D}'(\R^3,\R^3)$ such that there exists a Lipschitz one-to-one function 
$\gamma\::\: \T^1\to \R^3$ verifying
$$
T(X) = \int_{\T^1}X(\gamma(s))\cdot\gamma'(s)\, ds, \qquad \forall X \in \mathcal{D}(\R^3,\R^3).
$$ 
The length of a simple closed oriented curve $T$, denoted by $L(T)$, is given by
$
L(T) := \int_{\T^1} |\gamma'(s)|\, ds,
$
and we have the equality 
$
L(T) = \sup \{ T(X) \ : \ X \in
\mathcal{D}(\R^3,\R^3)\: , \: \|X\|_\infty\leq 1 \}, 
$
so that in particular $L(T)$ is independent of the choice of parametrization $\gamma.$

The set $\mathcal{T}$ of integral 1-currents in $\R^3$ without boundary is the set of vector valued distributions  
$T\in \mathcal{D}'(\R^3,\R^3)$
such that
$
T = \sum_{j\in \N} T_j
$
in $\mathcal{D}'(\R^3,\R^3)$ for a sequence $(T_j)_{j\in \N}$ of simple closed oriented curves in $\R^3$ such that
$
\sum_{j\in \N} L(T_j) < +\infty.
$ 
The mass of an integral 1-current in $\R^3$ without boundary $T\in \mathcal{T}$ is defined as
$
\|T\| := \sup \{ T(X) \ : \ X \in
\mathcal{D}(\R^3,\R^3)\: , \: \|X\|_\infty\leq 1 \}, 
$
and in we have 
$
\| T\| \leq \sum_{j\in \N} L(T_j)
$
whenever $T = \sum_{j\in \N} T_j$ for a sequence of simple closed oriented curves $(T_j)_{j\in \N}$ is
$\R^3.$

\noindent{\bf Oriented integral varifolds}. The set $\mathcal{V}$ of oriented integral 1-varifolds in $\R^3$ 
without boundary is defined\footnote{We emphasize that  ``integral" and ``without boundary" actually refer not to the
oriented varifold $V_t$ but to its first moment $T_{V_t}$. This is arguably an abuse of language, but it is convenient here. As a result, 
although the terminologies look similar, our definition of 
integral oriented varifold allows for non trivial measures with respect to $\xi$ variables, whereas the
definition of integral varifolds of Almgren and Allard doesn't.} as the set of finite non-negative Radon measures 
$V \in \mathcal{M}(\R^3\times S^2)$ whose first moment with respect to the $S^2$ variable 
$$
T_V\: : \: \mathcal{D}(\R^3,\R^3) \to \R,\ X \mapsto \int X(x)\cdot \xi dV(x,\xi)
$$
is an integral 1-current in $\R^3$ without boundary. The mass of $V\in \mathcal{V}$ is defined as
$
\|V\| := \sup \{ V(\psi) \ : \ \psi \in
\mathcal{D}(\R^3\times S^2,\R)\: , \: \|\psi\|_\infty\leq 1 \},
$ 
and in particular we always have the inequality
$
\|T_V\| \leq \|V\|.
$

\noindent{\bf Measurable and continuous families.} 
In the sequel, $I\subset \R$ denotes an interval 
such that $0\in I.$   
A family $(T_t)_{t\in I}$ of integral 1-currents in $\R^3$ without boundary is called continuous if
the map $t\mapsto T_t$ is continuous from $I$ to $\mathcal{D}'(\R^3,\R^3).$  A family $(V_t)_{t\in I}$ of
oriented integral 1-varifolds without boundary is called measurable if for every Borel subset $\mathcal{O} \subset
\R^3\times S^2,$ the map $t\mapsto V_t(\mathcal{O})$ is measurable on $I.$ 


\medskip
We are now in position to state:

\begin{defi}\label{def:genbf}
A measurable family $(V_t)_{t\in I}$ of oriented integral 1-varifolds in $\R^3$ without boundary
is called a generalized binormal curvature flow on $I$ if for any $X\in \mathcal{D}(\R^3,\R^3)$ the 
function $t\mapsto V_t(X\cdot\xi)$ is Lipschitz
on $I$ and satisfies
\begin{equation}\label{eq:cle}
\frac{d}{dt} \int  X \cdot \xi \, dV_t  = - \int D({\rm curl}(X))\, : \, \xi\otimes\xi\: dV_t
\end{equation}
for almost every $t\in I.$ 
\end{defi}

\begin{defi}\label{def:weakbf}
A continuous family $(T_t)_{t\in I}$ of integral 1-currents 
in $\R^3$ without boundary is called a weak binormal curvature flow on $I$ with initial datum $T_0$ if 
if there exists a generalized binormal curvature flow $(V_t)_{t\in I}$ on $I$ such that      
\begin{enumerate}
\item
The first moment $T_{V_t}$ of $V_t$ coincides with $T_t$ for every $t\in I.$  
\item
The mass $\|V_t\|$ satisfies 
$ \|V_t\| \leq \|T_0\|$ for every $t\in I.$
\end{enumerate}     
\end{defi} 

For a generalized binormal curvature flow $(V_t)_{t\in I}$ on $I$, we call the family of first moments
$(T_{V_t})_{t\in I}$ its family of associated undercurrents.

\begin{rem}\label{rem:lineaire}
i) Notice that Definition \ref{def:genbf} is linear in $V_t$. In particular, the sum of two generalized binormal
curvature flows is a generalized binormal curvature flow.  Also, if $(T_t^1)_{t\in I}$ and $(T_t^2)_{t\in I}$ 
are two weak binormal curvature flows with initial data $T_0^1$ and
$T_0^2$ respectively, and if moreover $\|T_0^1+T_0^2\| = \|T_0^1\| + \|T_0^2\|$, then 
$(T_t^1+T_t^2)_{t\in I}$ is a weak binormal curvature flow with initial datum $T_0^1+T_0^2.$

ii) Notice also that Definition \ref{def:genbf} only involves, in terms of $V_t$, its first moment on the left-hand 
side of \eqref{eq:cle} and its second moment on the
right-hand side of \eqref{eq:cle}. As a result, a uniqueness or a Cauchy theory for generalized binormal curvature flows at the level of $V_t$ is ruled out a priori. Further possible pathologies of generalized binormal curvature flows are illustrated by examples that we 
present in Remark \ref{rem:pathological}, at the end of Section \ref{S5.2}. 

As we will see, the situation greatly improves for weak binormal curvature flows.  

iii) Finally observe that the equality \eqref{eq:cle} actually makes sense for a general measurable family of 
Radon measures $V_t \in \mathcal{M}(\R^3\times S^2)$. 
Since we know only of artificial such 
examples of ``diffuse'' flows, we have preferred to stick with the actual Definition \ref{def:genbf}. 

Note, however, that Theorems \ref{thm:wsuniqueness} and \ref{thm:stability} below, which establish
weak-strong uniqueness of weak binormal curvature flows together with a related stability result, do not require the full strength of the definition of weak binormal curvature flow. Indeed, the assumption
that the undercurrents $T_{V_t}$ be integral for every  $t$ is not used anywhere in these proofs.
\end{rem}

\medskip
In view of Lemma \ref{lem:weakform}, we immediately deduce

\begin{prop}[Consistency]\label{prop:consistency}
Let $\ell>0$ and $\gamma \::\: I\times \left(\R/\ell\Z\right) \to \R^3$ denote a smooth classical solution of the binormal curvature flow 
equation \eqref{eq:strongbfbis}. The family $(V_{\gamma,t})_{t\in I}$ defined by
$$
V_{\gamma,t}(\psi) := \int_0^\ell \psi(\gamma(t,s),\partial_s\gamma(t,s))\, ds \qquad \forall\, \psi\in
\mathcal{D}(\R^3\times S^2,\R),
$$
is a generalized binormal curvature flow on $I$, and the family $(T_{\gamma,t})_{t\in I}$ defined by  
$$
T_{\gamma,t}(X) := \int_0^\ell X(\gamma(t,s))\cdot \partial_s\gamma(t,s)\, ds \qquad \forall\, X\in
\mathcal{D}(\R^3,\R^3)
$$ 
is a weak binormal curvature flow on $I$ with initial datum $T_{\gamma,0}$ provided $\|T_{\gamma,0}\|=\ell.$ 
\end{prop}

\smallskip

An advantage of Definitions \ref{def:genbf} and \ref{def:weakbf} is that lead rather directly to 
an existence theory globally in time. 

\begin{thm}[Global existence]\label{thm:existence}
For any integral 1-current in $\R^3$ without boundary $T_0,$ there exist a weak binormal curvature flow 
$(T_t)_{t\in \R}$ on $\R$ with initial datum $T_0.$
\end{thm} 

Theorem \ref{thm:existence} is proved using an approximation argument and compactness properties. We present  
some of these intermediate steps now which, we believe, have their own independent interest. 

\begin{prop}\label{prop:estimee}
Let $(V_t)_{t\in I}$ be a generalized binormal curvature flow on $I$ and denote by $(T_{V_t})_{t\in I}$ 
be its family of associated undercurrents. There exists a universal constant $C>0$ such that for every 
$t_1\,,\,t_2 \in I$  we have the inequality
$$
d_{\mathcal{F}^*}(T_{V_{t_1}},T_{V_{t_2}}) \leq  
C \: \Big(\sup_{t\in I} \|V_t\|^\frac12\Big)\,|t_1-t_2|^\frac12,
$$
where, for $T,\tilde T\in \mathcal{T},$ 
$$
d^*_{\mathcal{F}}(T,\tilde T) := \sup \left\{  T(X) - \tilde T(X) \ : \  X\in
\mathcal{D}(\R^3,\R^3)\: , \: \|{\rm curl}(X)\|_\infty \leq 1 \right\}. 
$$
In particular, whenever $(T_t)_{t\in I}$ is a weak binormal curvature flow on $I$ with initial datum $T_0,$ 
$$
d_{\mathcal{F}^*}(T_{{t_1}},T_{{t_2}}) \leq  
C \|T_0\|^\frac12\,|t_1-t_2|^\frac12 \qquad \forall t_1,t_2\in I.
$$
\end{prop}

\begin{rem}\label{re:area}
In geometric terms, the quantity $d_{\mathcal{F}}^*(T,\tilde T)$ is exactly equal to the area of the
two-dimensional minimal surface whose boundary is given by $T-\tilde T$ (see e.g. \cite{Fed} 4.1.12). The distance
$d_{\mathcal{F}}^*$ is also much related to and actually slightly stronger than Whitney's flat
metric. (In fact  $d_{\mathcal{F}}^*$ can be thought of as a homogeneous flat metric.) 
It follows therefore from Proposition \ref{prop:estimee} that when 
$(T_t)_{t\in I}$ is a weak binormal curvature flow on $I$ or the family of undercurrents associated to a 
generalized binormal curvature flow uniformly bounded in mass,  
the map $t\mapsto T_{t}$ is H\"older continuous 
with exponent $\frac12$ from $I\subset \R$ to $\mathcal{T}$ equipped with Whitney's flat metric. 
\end{rem}

\begin{prop}\label{prop:fermeture}
For each $n\in \N,$ let $(V^n_t)_{t\in I}$ be a generalized binormal curvature flow on $I.$ Assume 
that $\sup_{n\in \N,\, t\in I}\|V_t^n\| < +\infty$ and that
$$
V_t^n\, dt \rightharpoonup V \quad \text{ in }\quad \mathcal{M}(\R^3 \times S^2 \times I).
$$
Then $V=V_t\,dt$ in $\mathcal{M}(\R^2\times S^2\times I)$ where $(V_t)_{t\in I}$ is a generalized binormal
curvature flow on $I.$ Moreover, for every $t\in I$
$$
T_{V_t^n} \rightharpoonup T_{V_t}  \quad\text{in}\quad \mathcal{D}'(\R^3,\R^3)
$$ 
as $n\to +\infty.$
\end{prop}

Proposition \ref{prop:fermeture} implies in particular that every sequence of smooth binormal flows with uniform mass bounds and  possibly highly oscillatory behavior converges, along subsequences, to a  generalized flow. Examples of such limits which are not weak binormal curvature flows are provided in Section \ref{S5.2}.

\begin{cor}\label{cor:fermeture}
For each $n\in \N,$ let $(T_t^n)_{t\in I}$ be a weak binormal curvature flow on $I$
with initial datum $T_0^n.$ Assume that for some $T_0 \in \mathcal{T}$ we have, as $n\to +\infty,$
$$
T_0^n \rightharpoonup T_0 \quad \text{in} \quad \mathcal{D}'(\R^3,\R^3) \qquad\text{and}\qquad 
\|T_0^n\| \to \|T_0\| \quad \text{in} \quad \R.
$$
 Then there exist a subsequence $(n_k)_{k\in \N}$, and a weak binormal curvature flow 
$(T_t)_{t\in I}$  on $I$ with initial datum $T_0$ such 
that, as $k\to +\infty,$
$$
T_t^{n_k} \rightharpoonup T_t \quad \text{in}\quad  \mathcal{D}'(\R^3,\R^3), 
$$ 
for every $t\in I.$
\end{cor}

The first part of Proposition \ref{prop:fermeture} follows directly from Proposition \ref{prop:estimee} and 
the Arzel\`a-Ascoli theorem applied for a suitable localized version of the flat metric. The provided
convergence is actually stronger than stated in Proposition \ref{prop:fermeture} or Corollary \ref{cor:fermeture}
 (see Section \ref{sect:3}). Theorem \ref{thm:existence} follows from Corollary \ref{cor:fermeture}   
and the fact that integral 1-currents in $\R^3$ without boundary can be suitably approximated by finite sums 
of smooth closed curves, for which global existence of solutions to \eqref{eq:strongbfbis} can be used in
conjunction with Proposition \ref{prop:consistency} and the linearity mentioned in Remark \ref{rem:lineaire}.   

\medskip

Uniqueness of weak binormal curvature flows for a given initial datum $T_0$ fails in general under Definition
\ref{def:weakbf}, and in particular it is necessary to consider a subsequence in the statement of 
Corollary \ref{cor:fermeture}. We believe however that Definition \ref{def:weakbf} is sufficiently strong
to eliminate unrealistic sources of non uniqueness, and that the remaining ones are probably intrinsic to 
any reasonable formulation of weak binormal curvature flows that requires self-intersections 
and collisions to possibly matter. A typical example of non unique evolution is provided by an initial datum 
consisting of the sum of two circles of different radii (or else living in
different planes) and that have exactly one intersection point. A first evolution is given by the sum of the 
independent evolutions of both circles, which are traveling wave solutions, and whose mutual distance will
indefinitely increase since their speeds differ as vectors. A second evolution is obtained by approximating
the initial datum by smooth simple closed curves $T_0^n$ and applying Corollary \ref{cor:fermeture}
to their classical evolutions according to equation \eqref{eq:strongbfbis}. In this second case, the solution
at any time is supported in a Lipschitz image of $\T^1$, and therefore necessarily differs from the first evolution.           

\smallskip

Still, we have

\begin{thm}[Weak-strong uniqueness]\label{thm:wsuniqueness}
Let $\ell>0$ and $\gamma \::\: I\times \left(\R/\ell\Z\right) \to \R^3$ denote a smooth classical solution of the binormal curvature flow 
equation \eqref{eq:strongbfbis}, and assume that for any $t\in I$, the curve  
$\gamma_t:=\gamma(t,\cdot)$ is without self-intersection.  Then the weak binormal curvature flow 
$(T_{\gamma,t})_{t\in I}$ provided by Proposition \ref{prop:consistency} is the unique weak binormal curvature 
flow on $I$ with initial datum $T_{\gamma,0}.$    
\end{thm}

As a matter of fact, we deduce Theorem \ref{thm:wsuniqueness} from a stronger quantitative estimate. To that
purpose, consider a compact subset $J\subset I$ containing $0$ and set
$$
r \equiv r(\gamma,J) := \frac{1}{2}\min_{t\in J} \min\left( \|\partial_{ss}\gamma(t,\cdot)\|_{\infty}^{-1} ,
r_s(t) \right)  > 0,
$$
where the security radius $r_s(t)$ is defined as the largest positive real number with the property that 
every point $x$ satisfying $d(x,\gamma_t)< r_s(t)$ has a unique closest point $P_t(x)$ on $\gamma_t.$   
Define then the vector field $X_{\gamma,r}$ on $ \R^3\times J$ by\footnote{The function 
$f(d^2(\cdot,\gamma_t))$ vanishes where $P_t$ is undefined, so that $X_{\gamma,r}$ 
is globally well-defined.}
\begin{equation}\label{eq:Xgamma}
X_{\gamma,r} (x,t) = f(d^2(x,\gamma_t)) \tau_t(P_t(x))
\end{equation}
where $\tau_t$ is the oriented unit tangent vector along $\gamma_t$ and
$$
f(d^2) = \left\{
\begin{array}{ll}
\displaystyle \big(1- \big(\tfrac{d}{r}\big)^2\big)^3, & \ \text{for } \,0\:\leq d^2 \leq r^2,\\
\displaystyle \ 0, & \ \text{for } d^2\geq r^2.
\end{array}\right.
$$

\begin{thm}[Control of instability]\label{thm:stability}
Let $T_0 \in \mathcal{T}$ and let $(T_t)_{t\in J}$ be a weak binormal curvature flow on $J$ with 
initial datum $T_{0}.$ Define the non-negative functions $F$ and $G$ on $J$ by\footnote{Notice that the 
definitions of $F$ and $G$ only depend on the first moments $T_{V_t}$ of 
$V_t$; therefore $F$ and $G$ are uniquely determined by $T_t=T_{V_t}$ and well-defined.}
$$
G(t):= \|T_0\| - \int X_{\gamma,r}(x,t) \cdot \xi \: dV_t(x,\xi) \geq  F(t) := \int \bigl(1 - X_{\gamma,r}(x,t) \cdot \xi  \bigr)\: dV_t(x,\xi) \geq 0.
$$
Then $G$ is Lipschitzian on $J$ and
$$
\left| \frac{d}{dt} G(t) \right| \leq K F(t) \leq K G(t)
$$
almost everywhere on $J$, where 
$K \equiv K\big(r(\gamma,J),\,\|\partial_{sss}\gamma\|_{L^\infty(J\times \T^1)}\big).$
\end{thm}

As noted earlier, this result, and hence Theorem \ref{thm:wsuniqueness} as well, remains true if we 
drop the assumption (contained in the definition of a generalized binormal curvature flow) that $T_{V_t}$ be an {\em integral} 1-current. 

The function $F$ which appears in the statement of Theorem \ref{thm:stability} may be
understood as a measure of the discrepancy between $\gamma_t$ and $T_t$. 
To get some insight on its geometric meaning, we express the integral 1-current $T_t$ as 
$T_t = (\Gamma_t,\theta_t,\xi_t)$, where $\Gamma_t$, $\theta_t$ and $\xi_t$ are respectively 
the geometrical support, the multiplicity and the orientation of $T_t$, and then define 
$\Gamma_t^{in}=\{ x\in \Gamma_t \ \text{s.t. } d(x,\gamma_t)<r\}$ 
and $\Gamma_t^{out} = \Gamma_t \setminus \Gamma_t^{in}.$ 
For $x\in \Gamma_t^{in}$ such that $\tau_t(P_t(x))\cdot \xi_t(x) \geq 0,$ we have
$1-X_{\gamma,r}(x,t)\cdot \xi_t(x)
\geq 1 - \tau_t(P(x))\cdot \xi_t(x)
= \frac{1}{2}|\tau_t(P_t(x))-\xi_t(x)|^2,$ while for $x\in \Gamma^{in}$ such that 
$\tau_t(P_t(x))\cdot \xi_t(x) < 0,$ we have
$1-X_{\gamma,r}(x,t)\cdot\xi_t(x) \geq 1 \geq  \frac{1}{2}|\tau_t(P_t(x))-\xi_t(x)|^2.
$ It follows in particular that
\begin{equation}\label{eq:domine1}
F(t) \geq \int_{\Gamma_t^{in}} \frac{1}{2}|\tau_t\circ P_t-\xi_t|^2\: \theta_t \,
d\mathcal{H}^1 + \int_{\Gamma_t^{out}}\theta_t\, d\mathcal{H}^1.
\end{equation}
In a different direction, for $x\in \Gamma_t$ we also have
$1-X_{\gamma,r}(x,t)\cdot \xi_t(x) 
\geq 1 - f(d^2(x,\gamma_t))
\geq \min(d^2(x,\gamma_t), r^2),$ 
from which it follows that 
\begin{equation}\label{eq:domine2}
F(t) \geq \int_{\Gamma_t} \min(d^2(\cdot,\gamma_t), r^2) \: \theta_t\,
d\mathcal{H}^1.
\end{equation}
Upper bounds on $F(t)$ therefore provide upper bounds on the right-hand sides of
\eqref{eq:domine1} and \eqref{eq:domine2}, which together therefore correspond to an 
$H^1$ or tilt excess type measure of the discrepancy between $\gamma_t$ and $T_t$. Notice
however that $T_t$ may have multiple components, some of which, of small
total length, could be located arbitrarily far from $\gamma_t$ even if $F(t)$ is small. 
We refer to \cite{JeSm2} for the additional information that can be derived from $F$ when $T_t$ is
itself a classical mean curvature flow for a parametrized curve.

\medskip

Going back to Theorem \ref{thm:existence}, we mention that, whereas it is not difficult to produce weak binormal curvature flows for which 
$\|T_t\| < \|T_0\|$ for $t$ in some interval of positive length, e.g. by collision and annihilation of circles of opposite speeds, we
do not know of any such example for a flow constructed as a limit of smooth flows of single curves (see Section
\ref{S5.1} and the notion of almost parametric flows). On the other
hand, we have not been able to prove the contrary either, nor the fact that the $\xi$ part of the measures 
$V_t$ are always reduced to single Dirac masses. We believe that it would be of interest to obtain further 
insight to these questions.

We also  would like to stress that we have only considered here weak binormal curvature flows for finite
mass currents. In view of the fact that the quantity $1-X_{\gamma,r}\cdot \xi$ involved in the definition
 of $F$ in Theorem \ref{thm:stability} is pointwise non negative, it is not unreasonable to expect that 
part of the analysis could be carried out as well for integral 1-currents of locally finite mass, at least 
under suitable assumptions on their behavior at infinity. Such an extension would be of particular interest to 
consider the special solutions that have been recently studied in a series of interesting works by 
V. Banica and L. Vega \cite{BaVe1,BaVe2}, using quite different methods, and which correspond to perturbations of 
an infinitely extended broken line.

To conclude this introduction, we mention that integral formulas of a nature somewhat similar to 
\eqref{eq:weakbf1} have been known and used in the past in related, yet very different, contexts 
including the mean curvature flow and the incompressible Euler equations. Notably, the works of Brakke \cite{Bra} and  Ilmanen \cite{Ilm} 
have established existence and in some cases weak-strong uniqueness for mean curvature flows in the frameworks
of integral varifolds and integral currents. Whereas we deal with a Hamiltonian flow rather than a gradient
flow, it turns out that the existence part is simpler here in some aspects.  We have voluntarily stressed some 
analogies between the two situations in the way we stated Definition \ref{def:genbf} and Definition \ref{def:weakbf}, 
in particular regarding Brakke's definition of varifold mean curvature flow \cite{Bra} and Ilmanen's definition of 
enhanced motion \cite{Ilm}.
Regarding the Euler equations, a related integral formula  has been  used by DiPerna and Majda \cite{DiMa} to define and study a class of measure-valued solutions, and a weak-strong uniqueness theorem in this framework has recently been established by Brenier, de Lellis, and Sz\'ekelyhidi \cite{BrDLSz}. 
Whereas there are some analogies between our work and that
of \cite{DiMa, BrDLSz}, probably reflecting the fluid dynamical roots of the
 binormal curvature flow, it seems difficult in practice to directly relate the two
approaches; as already noted, this has been an open problem since the work of Helmholtz in the 1850s.

\medskip
We present the proofs of Lemma \ref{lem:weakform} and Proposition \ref{prop:consistency} in Section \ref{sect:2},
of Proposition \ref{prop:estimee}, Proposition \ref{prop:fermeture}, Corollary \ref{cor:fermeture} 
and Theorem \ref{thm:existence} in
Section \ref{sect:3}, and of Theorem \ref{thm:wsuniqueness} and Theorem \ref{thm:stability} in Section
\ref{sect:4}. In Section \ref{sect:5}, we gather some additional results as well as
some examples and open questions.  
   

\section{Proofs of Lemma \ref{lem:weakform} and Proposition \ref{prop:consistency}}\label{sect:2}

\noindent
{\bf Proof of Lemma \ref{lem:weakform}.} We expand both hands of \eqref{eq:weakbf1} in coordinates
and use the convention of summation over repeated indices. Concerning the left-hand side of \eqref{eq:weakbf1},
 we first have
\begin{equation*}\begin{split}
\frac{d}{dt} \int_{\T^1} (X\circ \gamma)\cdot \partial_s\gamma\, ds 
&= \int_{\T^1} 
((\partial_iX^j) \circ \gamma) \partial_t \gamma^i  \partial_s\gamma^j\, ds +\int_{\T^1} 
(X^j\circ \gamma)\partial_{st}\gamma^j\, ds\\ 
&= \int_{\T^1} 
((\partial_iX^j) \circ \gamma) \left( \partial_t \gamma^i  \partial_s\gamma^j - \partial_s \gamma^i \partial_t \gamma^j\right)\, ds.
\end{split}
\end{equation*}
By definition of the vector product 
$$
\left( \partial_t \gamma^i  \partial_s\gamma^j - \partial_t \gamma^j \partial_s
\gamma^i\right) = 
\eps_{ijk} \left( \partial_t \gamma \times \partial_s \gamma\right)^k,
$$
where $\eps_{ijk}$ is the permutation symbol, so that
\begin{equation}\label{eq:ep1}
\frac{d}{dt} \int_{\T^1} (X\circ \gamma)\cdot \partial_s\gamma\, ds =
 \eps_{ijk}  \int_{\T^1}((\partial_iX^j) \circ \gamma)  \left(
\partial_t \gamma \times \partial_s \gamma\right)^k\, ds.
\end{equation}
Concerning the right-hand side of \eqref{eq:weakbf1}, we write in coordinates
$$
\int_{\T^1} D\left( {\rm curl} X\right)(\gamma(t,s)) : \left(
\partial_s\gamma(t,s) \otimes \partial_s\gamma(t,s)\right) \, ds 
= 
\int_{\T^1} 
((\partial_l ({\rm curl} X)^k)\circ \gamma) \partial_s\gamma^l \partial_s\gamma^k\, ds.
$$
By definition of the rotational and the chain rule, 
\begin{equation*}
((\partial_l ({\rm curl} X)^k)\circ \gamma) \partial_s\gamma^l \partial_s\gamma^k
 = 
\eps_{ijk} 
((\partial_{il}X^j)\circ \gamma)
\partial_s\gamma^l \partial_s\gamma^k
 = 
\eps_{ijk}
\partial_s((\partial_{i}X^j)\circ \gamma)) \partial_s\gamma^k
.
\end{equation*}
Integration by parts therefore yields
\begin{equation}\label{eq:ep2}
\int_{\T^1} D\left( {\rm curl} X\right)(\gamma(t,s)) : \left(
\partial_s\gamma(t,s) \otimes \partial_s\gamma(t,s)\right) \, ds = - 
\eps_{ijk}\int_{\T^1}  (\partial_{i}X^j)\circ \gamma) \partial_{ss}\gamma^k\, ds.
\end{equation} 
Finally, since \eqref{eq:strongbfbis} holds we have $\partial_t \gamma
\times \partial_s \gamma = \partial_{ss}\gamma$ and the conclusion then follows combining \eqref{eq:ep1} and 
\eqref{eq:ep2}.
\qed

\medskip

\noindent
{\bf Proof of Proposition \ref{prop:consistency}.}
It is nothing more than a rephrasing of Lemma \ref{lem:weakform} in the frameworks of Definition \ref{def:genbf}
and Definition \ref{def:weakbf}.\qed

\section{Proofs of Proposition \ref{prop:estimee} and \ref{prop:fermeture}, Corollary \ref{cor:fermeture} and Theorem \ref{thm:existence}}\label{sect:3}

The point of the next proof is to interpolate between uniform bounds on 
$\|T_{V_t}\|$ and the Lipschitz continuity of $t\mapsto T_{V_t}$ with respect to a weak norm (roughly speaking, the norm dual to $\| D( {\rm curl}\,X)\|_\infty$),
implicit in the definition of a generalized binormal curvature flow.

\medskip

\noindent{\bf Proof of Proposition \ref{prop:estimee}.}
Let $X\in \mathcal{D}(\R^3,\R^3)$ and $t_1 \neq t_2 \in I.$ Let $\eps>0$ whose actual value will be determined
at the end of the proof, and set $\rho_\eps(x):=\eps^{-3}\rho(x/\eps)$, where $\rho(x)=\zeta(|x|)$ is a 
fixed non negative radially symmetric function in $\mathcal{D}(\R^3,\R)$, compactly supported in $B(0,1)$, and 
 such that $\int \rho =1.$ Define $X_\eps := \rho_\eps * X.$  
We have
\begin{equation}\label{eq:cgv0}
T_{V_{t_1}}(X) - T_{V_{t_2}}(X) = T_{V_{t_1}}(X-X_\eps) - T_{V_{t_2}}(X-X_\eps) + T_{V_{t_1}}(X_\eps) - T_{V_{t_2}}(X_\eps).
\end{equation}
We first estimate, in view of Definition \ref{def:genbf},
\begin{equation}\label{eq:cgv1}\begin{split}
T_{V_{t_1}}(X_\eps) - T_{V_{t_2}}(X_\eps) &= - \int_{t_1}^{t_2} \int_{\R^3\times S^2} 
D({\rm curl}(X_\eps))(x)\, : \, \xi\otimes\xi\: dV_t(x,\xi)\,dt\\
&\leq 3 |t_2-t_1| \|D({\rm curl}(X_\eps))\|_\infty \Big(\sup_{t\in I}\|V_t\|\Big)\\
&\leq 3 |t_2-t_1| \frac{C_1}{\eps}\|{\rm curl}(X)\|_\infty \Big(\sup_{t\in I}\|V_t\|\Big),
\end{split}
\end{equation}
where $C_1:=\int|\nabla\rho|<+\infty$ is a fixed constant. 
Next, for any $x\in \R^3$ and $j\in 1,2,3,$ we write 
$$
X^j(x)-X_\eps^j(x) = \int_{0}^\eps \eps^{-3} \zeta(\frac{r}{\eps}) \big( \int_{\partial B(x,r)} [X^j(y)-X^j(x)] 
d\mathcal{H}^2\big)\,dr. 
$$ 
For each $r>0,$ we expand
\begin{equation*}\begin{split}
\int_{\partial B(x,r)} [X^j(y)-X^j(x)]
d\mathcal{H}^2 &
= \int_{\partial B(x,r)}\int_0^1 \nabla X^j(sy+(1-s)x) \cdot (y-x) \, ds\, d\mathcal{H}^2\\
&= \int_0^1 \int_{\partial B(x,sr)} \nabla X^j(z) \cdot \frac{(z-x)}{rs}  d\mathcal{H}^2\frac{r}{s^2}\, ds\\
&= \int_0^1 \int_{B(x,sr)} \Delta X^j(z)\, dz \, \frac{r}{s^2}\, ds\\
&
= \int_{\R^3} \Delta X^j(z) k_{r}(|z-x|)\,dz,
\end{split}
\end{equation*}
where $k_r(\tau) = \int_{\tau/r}^{\max(\tau/r,1)} \frac{r}{s^2}\,ds= r(\frac{r}{\tau}-1)^+.$ It follows that
$$
X(x)-X_\eps(x) = K_\eps * \Delta X, \qquad\text{where }
K_\eps(y) := \eps^{-3}\int_0^\eps\zeta(r/\eps)k_r(y)\,dr.
$$
Hence, for $i=1,2$ and summing over repeated indices, we obtain
\begin{equation}\label{eq:cgv2}\begin{split}
T_{V_{t_i}}(X-X_\eps) &= T_{V_{t_i}} ( K_\eps * \Delta X) = T_{V_{t_i}}\big(K_\eps * \big(\nabla {\rm div}(X) + {\rm curl}\,{\rm
curl}(X)\big)\big)\\
&= T_{V_{t_i}} \big( \nabla \big( K_\eps * {\rm div}(X)\big)\big) +
\eps_{jk\ell}\,T_{V_{t_i}}\big( \partial_jK_\eps * \big({\rm curl}(X)\cdot e_k\big)e_\ell\big)\\
&\leq  0 + 6 \Big(\sup_{t\in I}|\|V_t\|\Big) \|DK_\eps\|_{1}\|{\rm curl}(X)\|_{\infty},
\end{split}
\end{equation}
where we have used the fact that $T_{V_{t_i}}$ is boundary free. Inspection of $K_\eps$ yields the estimate
$\|DK_\eps\|_{1} \leq \frac{C_2}\eps$ where $C_2>0$ depends only on $\rho$, and therefore, 
from \eqref{eq:cgv0},\eqref{eq:cgv1},\eqref{eq:cgv2} we
deduce
$$
T_{V_{t_1}}(X)-T_{V_{t_2}}(X) \leq \left( 3\frac{C_1}{\eps}|t_2-t_1| + 6C_2\eps \right)  \|{\rm curl}(X)\|_{\infty}
\Big(\sup_{t\in I}\|V_t\|\Big). 
$$
The conclusion follows choosing $\eps:=|t_2-t_1|^\frac12$ and $C:=3C_1+6C_2.$\qed 

\medskip

\noindent{\bf Proof of Proposition \ref{prop:fermeture}.}
First, it follows from the convergence $V_t^ndt\rightharpoonup V$ that
$$
V(\R^3 \times S^2 \times (a,b)) \leq \Big(\sup_{n\in \N,\ t\in I}\|V_t^n\|\Big) |b-a|,\qquad\forall a,b \in I,
$$
and therefore we may disintegrate $V$ as $V = V_t\, dt$ where the measurable family of non negative
Radon measures $(V_t)_{t\in I}$ on $\R^3\times S^2$, uniquely defined for almost every $t\in I$, satisfies 
\begin{equation}\label{eq:athena0}
\sup_{t\in I} \|V_t\| \leq  \liminf_{n\to +\infty}\sup_{t\in I}\|V^n_t\|.
\end{equation}

\smallskip

Next, for $m\geq 1$ and $T, \tilde T \in  \mathcal{T}$, set
$$
d_{\mathcal{F},m}(T,\tilde T) := 
\sup \Big\{T(X) - \tilde T(X) \:  , \, \|X\|_\infty \leq 1\,,\, \|{\rm curl}(X)\|_\infty \leq 1\,,
\,{\rm supp}(X)\subset B(0,m) \Big\}, 
$$
where $X\in \mathcal{D}(\R^3,\R^3),$ and define
$$
d_{\mathcal{F},{\rm loc}}(T,\tilde T) := \sum_{m=1}^{+\infty} 2^{-m} \frac{d_{\mathcal{F},m}(T,\tilde
T)}{d_{\mathcal{F},m}(T,\tilde T)+1}.  
$$ 
By the Federer and Fleming compactness theorem (see e.g. \cite{Fed} 4.2.17), for every $R>0$ the
set $Y := \{ T \in \mathcal{T}\::\: \|T\|\leq R\}$ equipped with the
metric $d_{\mathcal{F},{\rm loc}}$ is compact.  In the sequel, we fix $R:=\sup_{n\in \N,\ t\in I}\|V_t^n\|.$
In view of the inequality $\|T_{V^n_t}\|\leq \|V^n_t\|$, the definition of $R$ and Proposition \ref{prop:estimee},
it follows that the sequences of maps $t\mapsto T_{V^n_t}$, $n\in \N$, is equibounded and equicontinous in
$\mathcal{C}(I,Y).$ By Arzel\`a-Ascoli theorem, we infer that there exists a subsequence $(n_k)_{k\in \N}$ and a 
family $(T_t)_{t\in I}$ in $\mathcal{C}(I,Y)$ such that $t\mapsto T_t^{n_k}$ converge to 
$t\mapsto T_t$ in $\mathcal{C}(J,Y)$ as $k\to +\infty$
for any compact subset $J\subset I.$

\smallskip

Let $h\in \mathcal{D}(I,\R)$ and $X\in \mathcal{D}(\R^3,\R^3)$ be given. On one side we have
$$
\lim_{k\to +\infty} \int\int h(t) X(x)\cdot \xi dV_t{^n_{k}}\,dt = \lim_{k\to +\infty} \int h(t)
T_{V_t^{n_{k}}}(X) \,dt = \int h(t)
T_t(X) \,dt 
$$
and on the other side we also have
$$
\lim_{k\to +\infty} \int\int h(t) X(x)\cdot \xi dV_t{^n_{k}}\,dt =  \int\int h(t) X(x)\cdot
\xi dV_t\,dt =  \int h(t) T_{V_t}(X) dt.
$$
It follows from those last two equalities and the du Bois-Reymond lemma that $T_{V_t}(X) = T_t(X)$ for almost every $t\in I.$   
Considering a countable family of vector fields $X$ in $\mathcal{D}(\R^3,\R^3)$, dense in $\mathcal{D}(\R^3,\R^3)$
for the topology of uniform convergence, it follows next that $T_{V_t} = T_t$ in $\mathcal{D}'(\R^3,\R^3)$
for almost every $t\in I.$ In turn, this implies that the only cluster point of the family 
$t\mapsto T_t^n$ in $\mathcal{C}(J,Y)$ is given by $t\mapsto T_t$, and therefore that the convergence of
$t\mapsto T_t^n$ to $t\mapsto T_t$ in $\mathcal{C}(J,Y)$ holds without need to take a subsequence. Finally, we
 redefine $(V_t)_{t\in I}$ for a negligible set of $t$ in such a way that  
$T_{V_t} = T_t$ in $\mathcal{D}'(\R^3,\R^3)$ now holds for all $t\in I,$ and that \eqref{eq:athena0} is still
valid.

It remains to verify that $(V_t)_{t\in I}$ is a generalized binormal curvature flow.  
Let thus $X\in \mathcal{D}(\R^3,\R^3).$ For each $m\in 
\N,$ by assumption the function $t\mapsto T_t^{m}(X)$ is Lipschitz on $I$ and
$$
\frac{d}{dt} T_t^{m}(X) =  - \int D({\rm curl}(X))\, : \, \xi\otimes\xi\: dV_t^{m}
$$ 
for almost every $t\in I.$ In particular, $ \| \frac{d}{dt} T_t^{m}(X) \|_\infty \leq C(X) 
\Big( \sup_{n\in \N,\ t\in I} \|V_t^n\| \Big)$ depends possibly on $X$ but not on $m.$ Since the function $t\mapsto T_t(X)$ is the pointwise 
limit of the functions  $t\mapsto T_t^{m}(X)$ as $m\to +\infty$, the previous estimate implies that 
$t\mapsto T_t(X)$ is Lipschitz
on $I.$ For any $h\in \mathcal{D}(I,\R)$, passing to the limit in the equality
$$
\int  T_t^{m}(X) \,h'(t)\, dt =  \int\int D({\rm curl}(X))\, : \, \xi\otimes\xi\: dV_t^{m}\,h(t)\, dt,
$$
we obtain
$$
\int \int  X\cdot \xi dV_t\,h'(t)\,dt = \int  T_t(X)h'(t) \, dt =  \int\int D({\rm curl}(X))\, : \, \xi\otimes\xi\:
dV_t\,h(t)\,dt,
$$
and since $t\mapsto T_t(X)$ is Lipschitz this finally implies that 
$$
\frac{d}{dt} \int X\cdot \xi dV_t  =  - \int D({\rm curl}(X))\, : \, \xi\otimes\xi\: dV_t,
$$
for almost every $t\in I.$ \qed

\noindent{\bf Proof of Corollary \ref{cor:fermeture}.}
For each $n\in \N,$ let $(V_t^n)_{t\in I}$ be a generalized binormal curvature flow whose family of undercurrents
 is given by $(T_t^n)_{t\in \N}$ and such that $\sup_{t\in I}\|V_n^t\| \leq \|T_0^n\|.$ In view of the assumption
$\|T_0^n\|\to \|T_0\|$, we infer that $\sup_{n\in \N,\ t\in I} \|V_t^n\| < +\infty.$ By the de la Vall\'ee Poussin
theorem, there exist a subsequence $(n_k)_{k\in \N}$ and a non negative Radon measure $V\in \mathcal{M}(\R^3\times
S^2\times I)$ such that $V^{n_k}_tdt\rightharpoonup V$ in $\mathcal{M}(\R^3\times
S^2\times I)$ as $k\to +\infty.$  The conclusion then follows from Proposition \ref{prop:fermeture} and the inequality
$$
\sup_{t\in I}\|V_t\| \leq \liminf_{k\to +\infty}\sup_{t\in I} \|V^{n_k}_t\| \leq \liminf_{k\to
+\infty}\|T_0^{n_k}\| = \|T_0\|.
$$  

\noindent{\bf Proof of Theorem \ref{thm:existence}.}
We proceed by approximation. Let $T_0 \in \mathcal{T}.$ By
Federer approximation theorem \cite{Fed} 4.2.20, there exist a sequence $(T_0^n)_{n\in \N}$ in 
$\mathcal{T}$ such that $T_0^n \rightharpoonup T_0$ in $\mathcal{D}'(\R^3,\R^3)$ and $\|T_0^n\| \to \|T_0\|$ 
in $\R$ as $n\to +\infty,$ and such that for each $n\in \N$ $T_0^n$ has the following structure: there exist 
a finite number of smooth closed oriented curves $(\gamma^n_{j,0})_{j\in J(n)}$ in $\R^3$ such that 
\begin{equation}\label{eq:pastrop}
T_0^n = \sum_{j\in J(n)}T_{\gamma^n_{j,0}} \qquad\text{and}\qquad 
\|T_0^n\| = \sum_{j\in J(n)}\|T_{\gamma^n_{j,0}}\|. 
\end{equation}
For each $n\in \N$ and $j\in J(n),$ let $\gamma^n_j$ denote the global classical solutions of equation\footnote{Existence of classical solutions of \eqref{eq:strongbfbis} for smooth data is well-known;
one of the earlier proofs is given in \cite{SuSuBa}, in a slightly different setting.}
\eqref{eq:strongbfbis} with initial data $\gamma^n_{j,0}$
and set $\gamma^n_{j,t}:= \gamma^n_j(t,\cdot).$ By Proposition \ref{prop:consistency}, Remark \ref{rem:lineaire}, 
and \eqref{eq:pastrop}, we infer that for each $n\in \N$ the map $t\mapsto T_t^n := \sum_{j\in J(n)}T_{\gamma_{j,t}^n}$ defines a weak binormal 
curvature flow with initial datum $T_0^n.$ The conclusion then follows from
Corollary \ref{cor:fermeture}.\qed

\section{Proofs of Theorem \ref{thm:stability} and Theorem \ref{thm:wsuniqueness}}\label{sect:4}

We first prove the following key estimate\footnote{A very similar estimate, with a nearly identical proof, is also presented in 
our companion paper \cite{JeSm2}, which is more suitable to binormal curvature flows in parametric form only.}

\begin{prop}\label{prop:core}
Assume the hypotheses of Theorem \ref{thm:wsuniqueness}.
For any $\xi_0\in S^2\subset \R^3$, 
the estimate 
\begin{equation}\label{eq:waou}
\bigl|\partial_t X_{\gamma,r} \cdot \xi_0 - D({\rm curl} X_{\gamma,r}) : (\xi_0 \otimes \xi_0)  \bigr|   \leq  K \left( 1 -
X_{\gamma,r} \cdot \xi_0\right)
\end{equation}
holds on $\R^3\times J,$ where the vector field
$X_{\gamma,r}$ was  defined in \eqref{eq:Xgamma} and 
$$
K \equiv K := \frac{54}{r^2} + 14 \|\partial_{sss}\gamma\|_{L^\infty(J\times \T^1)}.
$$
\end{prop}
\begin{proof}
First notice that since $f$ vanishes otherwise, we may restrict our attention to points $(x_0, t_0)\in \R^3\times J$
such that $d(x_0,\gamma_{t_0})\leq r.$ Let $s_0\in \T^1$ be uniquely defined by $P_{t_0}(x_0)=\gamma(t_0,s_0).$ 
In particular, we have
\begin{equation}\label{eq:proche1} 
\big| x_0 - \gamma(t_0,s_0) \big| \big|\partial_{ss}\gamma(t_0,s_0)\big| \leq 1/2 
\end{equation}
and
\begin{equation}\label{eq:ortho1}
\big( x_0 - \gamma(t_0,s_0)\big) \cdot \partial_s \gamma(t_0,s_0)=0.
\end{equation}

The mapping $\Psi : \R^3 \times J \times \T^1 \to \R$
$$
(x,t,s) \mapsto \big( x - \gamma(t,s)\big) \cdot \partial_s \gamma(t,s) 
$$
satisfies $\Psi(x_0,t_0,s_0)=0$ and
\begin{equation}\label{eq:ibm1}
\partial_s \Psi(x_0,t_0,s_0) = - |\partial_s\gamma(t_0,s_0)|^2 +\big( x_0 - \gamma(t_0,s_0)\big) \cdot
\partial_{ss}\gamma(t_0,s_0) \leq -1/2 , 
\end{equation}
where we have used \eqref{eq:arc} and \eqref{eq:proche1} for the last inequality. From the implicit function
theorem, we infer that there exist an open neighborhood $\mathcal{U}$ of $(x_0,t_0)$ in $\R^4,$ 
and a smooth function $\zeta\,:\, \mathcal{U}\to \R$ such that
\begin{equation}
\Psi(x,t,\zeta(x,t)) = 0 \qquad \forall (x,t) \in \mathcal{U}.
\end{equation}   
We may assume that $\mathcal{U}\subset \{ (x,t) : d(x, \gamma_t)< \frac 32 r\}$, so that 
$P_t(x)$ is defined for $(x,t)\in \mathcal U$.
By uniqueness of the nearest-point projection, we therefore infer that 
$$
P_t(x) = \gamma(t,\zeta(x,t)) \qquad \forall (x,t) \in \mathcal{U},
$$
and also that
\begin{equation}\label{eq:defX}
X_{\gamma,r}(x,t) = f\Big(\big|x-\gamma(t,\zeta(x,t))\big|^2\Big) \partial_s\gamma(t,\zeta(x,t)),\qquad \forall (x,t) \in \mathcal{U}
\end{equation}
and finally that 
\begin{equation}\label{eq:ibm2}
\rho(x,t):= 1 - \big( x - \gamma(t,\zeta(x,t))\big) \cdot
\partial_{ss}\gamma(t,\zeta(x,t))  >0 \qquad \mbox{ in }\mathcal U.
\end{equation}

We fix some notation to keep subsequent expressions of reasonable size. 
 For a function $Y$ with values in $\R^3,$ and $i\in\{1,2,3\},$ we write $Y^i$ to denote the i-th
component of $Y.$ We write  $d^2$ to denote the function 
$(x,t)\mapsto |x-\gamma(t,\zeta(x,t))|^2$, $\gamma(\zeta)$ to denote the function $(x,t)\mapsto \gamma(t,\zeta(x,t))$, and similarly for
$\partial_t\gamma(\zeta),$  $\partial_{ts}\gamma(\zeta),$ $\partial_s\gamma(\zeta),$ $\partial_{ss}\gamma(\zeta)$ and
$\partial_{sss}\gamma(\zeta).$ When it does not lead to possible confusion, we also denote by $x$ the function
$(x,t)\mapsto x.$  Each of these functions is defined on $\mathcal{U}.$

{\bf Step 1: First computation of \mb$D({\rm curl} X) : (\xi_0 \otimes \xi_0).$\mn}    
Differentiating \eqref{eq:defX} we obtain, pointwise on $\mathcal{U}$ and for $i,j \in\{1,2,3\},$  
\begin{equation}\label{eq:djXgamma}
\partial_{j}X^i = \partial_{j}(f(d^2) )\partial_s \gamma(\zeta)^i
+ f(d^2) \, \partial_{ss}\gamma(\zeta)^i \, \partial_{j}\zeta
\end{equation}
for the space derivatives, and 
\begin{equation}
\label{eq:dtXgamma}
\partial_{t}X^i = \partial_{t}(f(d^2) )\, \partial_s \gamma(\zeta)^i
+ f(d^2) \, \big[ \partial_{ss}\gamma(\zeta)^i \, \partial_{t}\zeta + 
\partial_{ts}\gamma(\zeta)\big] 
\end{equation}
for the time derivative. Also, for $i,j,\ell \in\{1,2,3\},$ 
\begin{equation}\begin{split}
\partial_{\ell j} X^i =  & 
\partial_{\ell j} (f(d^2 ))\, \partial_s \gamma(\zeta)^i
+\partial_{ss}\gamma(\zeta)^i \big[ \partial_{\ell} (f(d^2)) \, \partial_{j}\zeta + \partial_{j} (f(d^2)) \,
\partial_{\ell}\zeta\big]\nonumber
\\
&
+ f(d^2) \, \partial_{sss}\gamma(\zeta)^i \,\partial_{\ell}\zeta \,\partial_{j}\zeta
+ f(d^2) \, \partial_{ss}\gamma(\zeta)^i \,\partial_{\ell j} \zeta.
\label{eq:dldjXgamma}
\end{split}\end{equation}
In particular, we may write
\begin{equation}
D({\rm curl} X) : (\xi_0 \otimes \xi_0) 
\ =: \  A 
 \ = \ A_1 + A_2 + A_3 + A_4, 
\label{eq:Dcurl1}\end{equation}
where
\begin{equation}\label{eq:ugly1}
\begin{aligned}
A_1& :=  \ep_{ijk} \partial_{\ell i} (f(d^2 ))\, \partial_s \gamma(\zeta)^j \,\xi_0^k \xi_0^\ell \,  ,\\
A_2&:= 
\ep_{ijk} 
\partial_{ss}\gamma(\zeta)^j \big[ \partial_{\ell} (f(d^2)) \, \partial_{i}\zeta  + \partial_{i} (f(d^2)) \
\partial_{\ell}\zeta  \big]\xi_0^k \xi_0^\ell \, ,
\\
A_3&
:= \ep_{ijk} \, f(d^2)\, \partial_{sss}\gamma(\zeta)^j \,\partial_{\ell} \zeta  \, \partial_{i}\zeta \xi_0^k \xi_0^\ell \, , \\
A_4&
:=  \ep_{ijk} f(d^2) \ \partial_{ss}\gamma(\zeta)^j \,\partial_{\ell i}\zeta
\xi_0^k\xi_0^\ell,
\end{aligned}
\end{equation}
in which $\ep_{ijk}$ is the Levi-Civita symbol and we sum over repeated indices. 

\medskip

{\bf \mb Step 2: Expressing derivatives of $\zeta$ in terms of $\gamma$.\mn} 
Recall that by definition of $\zeta$, we have
\begin{equation}\label{eq:sigma1}
(x - \gamma(t,\zeta(x,t)) )\cdot \partial_s \gamma(t,\zeta(x,t)) = 0
\end{equation}
for every $(x,t)\in \mathcal{U}.$ For $j\in \{1,2,3\},$ differentiating \eqref{eq:sigma1} 
with respect to $x_j$  and using \eqref{eq:arc} we find 
\begin{equation}\label{eq:sigma1.5}
\partial_s \gamma^j (\zeta)- \partial_{j}\zeta + (x-\gamma(\zeta))\cdot
\partial_{ss}\gamma(\zeta) \ \partial_{j}\zeta = 0 .
\end{equation}
In view of \eqref{eq:ibm2}, we may rewrite \eqref{eq:sigma1.5} as
\begin{equation}\label{eq:sigma2}
\partial_{j}\zeta  = \frac 1 {\rho} \partial_s \gamma(\zeta).
\end{equation}
For $\ell\in \{1,2,3\},$ differentiating \eqref{eq:sigma1.5} with respect to $x_\ell$ and using 
\eqref{eq:ibm1}, we obtain
\begin{equation}\label{eq:sigma3}
\partial_{\ell j}\zeta
= \frac 1 \rho \Big(
\partial_{ss}\gamma(\zeta)^j \frac {\partial_s \gamma(\zeta)^\ell}\rho 
+ 
\partial_{ss} \gamma(\zeta)^\ell \frac {\partial_s \gamma(\zeta)^j}\rho 
+
(x-\gamma(\zeta))\cdot \partial_{sss}\gamma(\zeta) \frac
{\partial_s\gamma(\zeta)^j \partial_s\gamma(\zeta)^\ell}{\rho^2}
\Big).
\end{equation}
Finally, differentiating \eqref{eq:sigma1} with respect to $t$ we obtain
\begin{equation}
\partial_t\zeta \ = \ 
\frac 1 \rho( - \partial_t\gamma(\zeta) \cdot \partial_s \gamma(\zeta) + (x -
\gamma(\zeta))\cdot\partial_{ts} \gamma(\zeta)).
\label{eq:sigma4}\end{equation}
In particular, taking into account \eqref{eq:strongbfbis} it follows from \eqref{eq:sigma4} that, at the point $(x_0,t_0)$,
\begin{equation}\label{eq:sigma4point}
\partial_t\zeta \ = \ 
\frac {1}{\rho} (x - \gamma(\zeta))\cdot\ \big(\partial_s \gamma(\zeta)\times
\partial_{sss} \gamma(\zeta)\big).
\end{equation}

\medskip

{\bf \mb Step 3: Expressing derivatives of $d^2$ in terms of $\gamma.$\mn} 
In view of the definition of $d^2$, we have for $j\in \{1,2,3\},$  
\begin{equation}\label{eq:Ddsquared}
\partial_{j} d^2  = 2(x - \gamma(\zeta))^j - 2(x - \gamma(\zeta))\cdot
\partial_s\gamma(\zeta)\ 
\partial_{j}\zeta \  = \  2(x - \gamma(\zeta))^j, 
\end{equation}
where the last equality follows from \eqref{eq:sigma1}. 
For $\ell\in \{1,2,3\},$ differentiating \eqref{eq:Ddsquared} with respect to $x_\ell$ and using \eqref{eq:sigma1.5}, 
we obtain  
\begin{equation}\label{eq:D2dsquared}
\partial_{\ell j} d^2 = -2\big( \delta_{j\ell } - \partial_s \gamma(\zeta)^j\partial_{\ell}\zeta\big) =
-2\Big(\delta_{j\ell} - \frac {\partial_s \gamma(\zeta)^j\, \partial_s \gamma(\zeta)^\ell}\rho\Big),
\end{equation}
where $\delta_{j\ell}$ is the Kronecker symbol. Also from the definition of $d^2$, we have 
\begin{equation}\label{eq:dtdsquared}
\partial_{t} d^2 =  - 2(x - \gamma(\zeta))\cdot (\partial_t\gamma(\zeta) + \partial_s \gamma(\zeta) \,  \partial_t\zeta).
\end{equation}
In particular, taking into account \eqref{eq:sigma1} and \eqref{eq:strongbfbis} it follows from \eqref{eq:dtdsquared} that, at the point $(x_0,t_0)$,
\begin{equation}\label{eq:dtdsquaredpoint}
 \partial_{t} d^2 =  - 2(x - \gamma(\zeta))\cdot (\partial_s\gamma(\zeta) \times \partial_{ss}
   \gamma(\zeta)).
\end{equation}

\medskip

{\bf \mb Step 4: A reduced expression for \mb$D({\rm curl} X) : (\xi_0 \otimes \xi_0).$\mn} We substitute,
in the terms $A_1, A_2, A_3$ and $A_4$ defined in Step 1, the expressions for the derivatives of 
$d^2$ and $\zeta$ which we obtained in Step 2 and Step 3. Some cancellations occur.

\smallskip

Examining $A_1$, we first expand:
\begin{equation*}\begin{split}
\partial_{\ell i} (f(d^2 )) &=   f''(d^2) \partial_{\ell}d^2 \,
\partial_{i}d^2 + f'(d^2) \partial_{\ell i} d^2\\
&=  4f''(d^2)
(x-\gamma(\zeta))^\ell \, (x-\gamma(\zeta))^i +\tfrac{2}{\rho}f'(d^2)\partial_s \gamma(\zeta)^\ell\,
\partial_s \gamma(\zeta)^i - 2f'(d^2)\delta_{\ell i},
\end{split}
\end{equation*}
where we have used \eqref{eq:Ddsquared} and \eqref{eq:D2dsquared} for the second equality. Next, we write
\begin{align*}
\ep_{ijk} 
(x-\gamma(\zeta))^\ell \, (x-\gamma(\zeta))^i
\partial_s \gamma(\zeta)^j\, \xi_0^k \xi_0^\ell &= \big(\ep_{ijk}(x-\gamma(\zeta))^i\partial_s
\gamma(\zeta)^j\xi_0^k\big)\big((x-\gamma(\zeta))^\ell \xi_0^\ell\big)\\ 
&= \big((x - \gamma(\zeta))\cdot (\partial_s\gamma(\zeta)\times \xi_0) \big)\big( (x - \gamma(\zeta)) \cdot
\xi_0\big).\\
\shortintertext{Similarly,} 
\ep_{ijk} \partial_s \gamma(\zeta)^\ell\,
\partial_s \gamma(\zeta)^i\, \partial_s \gamma(\zeta)^j\, \xi_0^k \xi_0^\ell &= \big(\partial_s \gamma(\zeta)\cdot
(\partial_s\gamma(\zeta)\times \xi_0) \big)\big( \partial_s\gamma(\zeta) \cdot \xi_0\big) = 0,\\
\shortintertext{and}
\ep_{ijk} \delta_{\ell i} \, \partial_s \gamma(\zeta)^j\, \xi_0^k \xi_0^\ell &=  \ep_{ijk} \xi_0^i \, \partial_s
\gamma(\zeta)^j\, \xi_0^k = \xi_0\cdot \big( \partial_s\gamma(\zeta)\times \xi_0\big) = 0.
\end{align*}
Hence,
\begin{equation}\label{eq:A1}
A_1 = 4f''(d^2)\big((x - \gamma(\zeta))\cdot (\partial_s\gamma(\zeta)\times \xi_0) \big)\big( (x - \gamma(\zeta)) \cdot
\xi_0\big). 
\end{equation}
In the same way as for $A_2$, \eqref{eq:sigma2} and \eqref{eq:Ddsquared} yield 
\begin{equation}\label{eq:A2}\begin{split}
A_2 = &\ \tfrac 2 \rho f'(d^2) \ep_{ijk} \xi_0^k \xi_0^\ell 
\partial_{ss}\gamma(\zeta)^j  \big( (x-\gamma(\zeta))^\ell \partial_s
\gamma(\zeta)^i  +(x-\gamma(\zeta))^i   \partial_s \gamma(\zeta)^\ell  \big)\\
= & \   
\tfrac 2 \rho f'(d^2)  \partial_{ss}\gamma(\zeta)\cdot  (\partial_s
\gamma(\zeta) \times \xi_0)  (\xi_0 \cdot (x-\gamma(\zeta))) \ +
 \\
& \ \tfrac 2 \rho f'(d^2) (x-\gamma(\zeta))\cdot(\partial_{ss}\gamma(\zeta) \times \xi_0) (\xi_0 \cdot
\partial_s\gamma(\zeta))
\\
 =&\!: \ A_{2,1} + A_{2,2}.
\end{split}\end{equation}
For $A_3,$ we invoke \eqref{eq:sigma2} to substitute $\partial_\ell \zeta$ and $\partial_i\zeta$ and obtain 
\begin{equation}\label{eq:A3}
A_3   =   \frac 1{\rho^2}f(d^2) (\partial_s\gamma(\zeta)\cdot \xi_0) \
\partial_s\gamma(\zeta)
\cdot(\partial_{sss}\gamma(\zeta)\times \xi_0).
\end{equation}
For $A_4$ finally, we invoke \eqref{eq:sigma3} to substitute $\partial_{\ell i}\zeta$ and obtain
\begin{equation}\begin{split}
A_4 =&\ \tfrac 1{\rho^2} f(d^2)\partial_{ss}\gamma(\zeta)\cdot\big(\partial_{ss}\gamma(\zeta)\times \xi_0\big)
(\partial_{s}\gamma(\zeta)\cdot \xi_0)
\ +\\
&\ \tfrac 1{\rho^2} f(d^2)  \partial_s \gamma(\zeta)\cdot
(\partial_{ss}\gamma(\zeta)\times \xi_0 )
(\partial_{ss}\gamma(\zeta)\cdot \xi_0) 
\ + \\
&\ \tfrac 1{\rho^3} f(d^2) 
((x-\gamma(\zeta))\cdot \partial_{sss}\gamma(\zeta))\, \partial_s
\gamma(\zeta)\cdot( \partial_{ss}\gamma(\zeta)\times \xi_0)(\partial_s \gamma(\zeta)\cdot \xi_0)
\\
=&\!:\ 0 + A_{4,1} + A_{4,2}.
\end{split}
\end{equation}

{\bf \mb Step 5: Computation of $\partial_t X \cdot \xi_0$.\mn}
We expand \eqref{eq:dtXgamma} as
\begin{equation}
\label{eq:dtXgammabis}
\partial_{t}X^i = f'(d^2)\partial_{t}d^2 \, \partial_s \gamma(\zeta)^i
+ f(d^2) \, \big[ \partial_{ss}\gamma(\zeta)^i \, \partial_{t}\zeta + 
\partial_{ts}\gamma(\zeta)\big]. 
\end{equation}
Therefore, {\it at the point} $(x_0,t_0),$ we obtain from \eqref{eq:strongbfbis}, \eqref{eq:sigma4} and
\eqref{eq:dtdsquaredpoint}
\begin{equation}\label{eq:dtXdotV}
\partial_t X \cdot \xi_0 =: B = B_1 + B_2 + B_3, 
\end{equation} 
where
\begin{equation}\label{eq:notsobad}
\begin{aligned}
B_1 
&:= -2  f'(d^2) ((x - \gamma(\zeta))\cdot (\partial_s \gamma(\zeta) \times
\partial_{ss}\gamma(\zeta))) \, ( \partial_s\gamma(\zeta) \cdot \xi_0)
\, , 
\\
B_2
&:= \tfrac 1 {\rho} f(d^2)\,  ((x-\gamma(\zeta))\cdot (\partial_s\gamma(\zeta)\times
\partial_{sss}\gamma(\zeta))) \,( \partial_{ss}\gamma(\zeta) \cdot \xi_0)\, ,
\\
B_3 
&:= f(d^2)  \, (\partial_s\gamma(\zeta) \times \partial_{sss}\gamma(\zeta))\cdot \xi_0\, . 
\end{aligned}
\end{equation}

{\bf \mb Step 6: Proof of Proposition \ref{prop:core} completed.\mn} 
We write, {\it at the point} $(x_0,t_0),$ 
\begin{equation}\label{eq:rassemble}\begin{split}
\bigl| B - A \bigr| &= \bigl|\partial_t X \cdot \xi_0 - D({\rm curl} X) : (\xi_0 \otimes \xi_0)  \bigr| \\
& \leq |A_1| + |A_{2,1}| + |A_{2,2} - B_1| +  |A_3 - B_3| + |A_{4,1}| + |A_{4,2}| + |B_2|,
\end{split}\end{equation}
and we will estimate each of the terms in the last line separately. 
We first observe the following elementary facts that hold at the point $(x_0,t_0)$ (when they involve functions):
$$
\begin{array}{lll}
a) &\ds |\xi_0^\perp| \, , \, |\xi_0|\, ,\, |\partial_s\gamma(\zeta)| \leq 1,  &(\text{indeed } \xi_0\in S^2 \text{ and
}\eqref{eq:arc} \text{ holds}), \\
b) &\ds |f'(d^2)| \leq 3/r^2,\ |f''(d^2)|\leq 6/r^4,  & (\text{this follows from the definition of }f),\\
c) &\ds \rho\geq 1/2, |1-1/\rho|\leq d/r, |1 - 1/\rho^2|\leq 3d/r, &(\text{from }\eqref{eq:proche1} \text{ and
the definition \eqref{eq:ibm2} of }\rho). 
\end{array}
$$
For convenience, set $\Sigma = \|\partial_{sss}\gamma\|_{L^\infty(J\times \T^1)}.$
Taking into account \eqref{eq:ortho1}, direct inspection yields  
\begin{equation}\label{eq:premieres}
|A_1| \le {24}d^2r^{-4}, \qquad |A_{2,1}| \le 6 d r^{-3} |\xi_0^\perp|, \qquad 
|A_{4,1}| \le \frac12 r^{-2}|\xi_0^\perp|^2,
\end{equation}
as well as
\begin{equation}\label{eq:secondes}
|A_{4,2}| \le  4 dr^{-1} \Sigma |\xi_0^\perp|,
\quad\quad 
|B_2| \le dr^{-1}\Sigma |\xi_0^\perp|.
\end{equation}
Next, we write  
\begin{equation}\label{eq:tierces}\begin{split}
| B_1 - A_{2,2}|  &= \big| \tfrac 2 \rho f'(d^2) (\partial_s\gamma(\zeta)\cdot \xi_0) 
((x-\gamma(\zeta))\times \partial_{ss}\gamma(\zeta) )\cdot
 ( \partial_s\gamma(\zeta) - \tfrac{\xi_0}{\rho} )\big|\\
&\le 6 dr^{-3} \big(|\partial_s\gamma(\zeta) - \xi_0| + dr^{-1}\big),
\end{split}\end{equation}
and
\begin{equation}\label{eq:quartes}\begin{split}
|B_3 - A_3| &=  \big| f(d^2)[(\partial_s\gamma(\zeta)\times \partial_{sss}\gamma(\zeta))\cdot \xi_0]( 1
 - \tfrac 1 {\rho^2}\xi_0\cdot \partial_s\gamma(\zeta))
\big|\\
&\le \Sigma |\xi_0^\perp| \ \big( (1- \partial_s\gamma(\zeta)\cdot \xi_0) +
dr^{-1}\big)\\
&\le \Sigma  \big( (1- \partial_s\gamma(\zeta)\cdot \xi_0) + dr^{-1}|\xi_0^\perp|\big).
\end{split}\end{equation}
It remains to bound $d$, $|\xi_0^\perp|$, $|\partial_s\gamma(\zeta) - \xi_0|$ and $|1- \partial_s\gamma(\zeta)\cdot \xi_0|$
in terms of $1-X_{\gamma,r}\cdot \xi_0.$ To that purpose, first recall from the definition of $f$, from the fact that
$|\xi_0|=|\partial_s\gamma(\zeta)| = 1$, and form the assumption $d\leq r$, that 
\begin{equation}\label{eq:dominun}
1 - X_{\gamma,r} \cdot \xi_0 \geq  1- f(d^2) \geq d^2r^{-2}.
\end{equation}
Also, if $X_{\gamma,r}\cdot \xi_0\geq 0$ then $1-X_{\gamma,r}\cdot \xi_0 \geq 1-\partial_s\gamma(\zeta)\cdot \xi_0,$ and if $X_{\gamma,r}\cdot \xi_0 <0$ then
$1-X_{\gamma,r}\cdot \xi_0 \geq 1 \geq (1-\partial_s\gamma(\zeta)\cdot \xi_0)/2.$ In any case, we have
\begin{equation}\label{eq:domindeux}
1 - X_{\gamma,r} \cdot \xi_0 \geq  \frac{1}{2} \left( 1 - \partial_s\gamma(\zeta)\cdot \xi_0\right).
\end{equation}
Finally, by Hilbert's projection theorem 
\begin{equation}\label{eq:domintrois}
|\xi_0^\perp|^2 \le |\xi_0 - \partial_s\gamma(\zeta)|^2 = 2\big(1 - \partial_s\gamma(\zeta)\cdot
\xi_0\big) \leq 4 \big(1-X_{\gamma,r}\cdot \xi_0\big).
\end{equation}
Inserting \eqref{eq:dominun}, \eqref{eq:domindeux}, or \eqref{eq:domintrois} in
\eqref{eq:premieres}-\eqref{eq:quartes}, and writing $x\preceq y$ for $x\leq y(1-X_{\gamma,r}\cdot \xi_0)$, we obtain
$$
|A_1|\preceq \frac{24}{r^2},\quad |A_{2,1}|\preceq \frac{12}{r^2},\quad |A_{4,1}|\preceq \frac{2}{r^2},
\quad |A_{4,2}|\preceq 8\Sigma,
$$
$$|B_2|\preceq 2\Sigma,\quad |B_1-A_{2,2}|\preceq
\frac{16}{r^2},\quad |B_3-A_3|\preceq 4\Sigma,  
$$
and summation according to \eqref{eq:rassemble} yields the claim.
\end{proof}

\medskip

\noindent{\bf Proof of Theorem \ref{thm:stability}.} Since the map 
$t\mapsto \|V_t\|$ is bounded on $J$, since $f$ is of class
$\mathcal{C}^2$ and since $D({\rm curl}(X_{\gamma,r}))$ is a continuous function, we infer
from Definition \ref{def:genbf} that the function $G$ is Lipschitz on $J$ and that 
$$
\frac{d}{dt}G(t) = -\frac{d}{dt}\int X_{\gamma,r}\cdot \xi \, dV_t =  - \int  \partial_t
X_{\gamma,r} \cdot \xi - D({\rm curl} X_{\gamma,r}) : (\xi \otimes \xi) \, dV_t.
$$
The conclusion follows directly from Proposition \ref{prop:core}.\qed

\medskip

\noindent{\bf Proof of Theorem \ref{thm:wsuniqueness}.}
Let $t\mapsto T_t$ be a weak binormal curvature flow on $J$ with initial datum $T_{\gamma,0}.$ By Theorem 
\ref{thm:stability} and Gronwall inequality, we infer that $G$ and $F$ vanish identically on any 
compact subinterval of $J$ containing $0$,  and therefore vanish on $J.$ Fix $t\in J.$ Since 
$X_{\gamma,r}(x,t)\cdot \xi = 0$ if and 
only if $x = \gamma(t,s)$ for some $s\in \T^1$ and $\xi=\partial_s\gamma(t,s),$ we deduce from the identity $F(t)=0$
that for $\mathcal{H}^1$-a.e. $x$ in the geometrical support of $T_t$ we have $x\in \gamma_t.$ It follows from the
Federer-Fleming constancy theorem \cite{Fed} 4.1.31 that $T_t = aT_{\gamma,t}$ for some $a\in \Z,$ and then from the
identity $F(t)=0$ that $a=1.$ \qed
    
\section{Additional results, examples and open questions}\label{sect:5}

\subsection{Control of average speed and conserved quantities.}\label{S5.1}

In general, the convergence stated in Proposition \ref{prop:fermeture} or Corollary \ref{cor:fermeture}, and involved in the
construction of a solution in Theorem \ref{thm:existence}, does not imply that there is no mass loss at infinity,
and it could be that $\|V_t\|$ is not constant in time. In the following, we present a sufficient condition
to rule out this possibility, and we deduce conservation of momentum and angular momentum in that case.  

\begin{defi}\label{def:paraweak}
A weak binormal curvature flow $(T_t)_{t\in I}$ is called almost parametric if there exists a sequence of binormal
curvature flows $(T_{\gamma^n,t})_{t\in I}$, $n\in \N$, associated to smooth solutions $(\gamma^n)_{n\in \N}$ 
of \eqref{eq:strongbfbis} according to Proposition \ref{prop:consistency}, such that
$$
\|T_0\| = \lim_{n\to +\infty} \|T_{\gamma^n,0}\|
$$
and
$$
T_{\gamma^n,t} \rightharpoonup T_t \qquad \text{in}\qquad \mathcal{D}'(\R^3,\R^3)
$$
for all $t\in I.$  
\end{defi} 

\begin{rem}\label{rem:ahbon}
i) It follows from Proposition \ref{prop:fermeture} that given any current $T_0$ associated to a Lipschitz 
function $\gamma_0 : \R/\ell\mathbb{Z} \to \R^3$ by the formula
$$
T_0(X) := \int_0^\ell X(\gamma_0(s))\cdot \partial_s\gamma_0(s)\, ds,\qquad\forall \ X\in \mathcal{D}(\R^3,\R^3),
$$
there exists an almost parametric binormal curvature flow with initial datum $T_0.$ \\
ii) From the convergence $T_{\gamma^n,t} \rightharpoonup T_t$ it follows that $T_t$ is compactly supported for
every $t\in I.$
\end{rem}

For a smooth solution $\gamma\: :\: I\times \R/\ell \Z \to \R^3$ of \eqref{eq:strongbfbis}, the momentum 
$P(\gamma(t,\cdot))$ and the angular momentum  $Q(\gamma(t,\cdot))$ defined respectively by 
$$
\begin{array}{ll}
& \displaystyle P(\gamma(t,\cdot)) := \int_0^\ell \gamma(t,s) \times \partial_s\gamma(t,s)\, ds,\\
& \displaystyle Q(\gamma(t,\cdot)) := \int_0^\ell \gamma(t,s) \times \Big(\gamma(t,s) \times
\partial_s\gamma(t,s)\Big)\, ds,
\end{array}
$$ 
are independent of time. 

Notice that 
$$
\begin{array}{l}
\displaystyle P(\gamma(t,\cdot))=(T_{\gamma,t}(X_1), T_{\gamma,t}(X_2),T_{\gamma,t}(X_3)),\\
\displaystyle Q(\gamma(t,\cdot))=(T_{\gamma,t}(Y_1), T_{\gamma,t}(Y_2), T_{\gamma,t}(Y_3)),
\end{array}
$$
where
the vector fields $X_1,X_2,X_3$ and $Y_1,Y_2,Y_3$ on $\R^3$ are given by $X_1:=(0,-x_3,x_2),$
$X_2:=(x_3,0,-x_1)$, $X_3:=(-x_2,x_1,0),$    
$Y_{1}:= (-x_2^2-x_3^2,x_1x_2,x_1x_3),$ $Y_{2}:=(x_1x_2,-x_1^2-x_3^2,x_2x_3),$ and
$Y_{3}:=(x_1x_3,x_2x_3,-x_1^2-x_2^2).$

\begin{defi}
Let $T$ be a compactly supported integral 1-current without boundary in $\R^3.$ The momentum of $T$,
denoted by $P(T)$, and the angular moment of $T$, denoted by $Q(T),$ are the vectors in $\R^3$ defined
by $P(T):=(T(X_1),T(X_2),T(X_3))$ and  $Q(T):=(T(Y_1),T(Y_2),T(Y_3)).$
\end{defi}

The sufficient condition which we rely on amounts to non vanishing of the momentum. 

\begin{prop}\label{prop:finitespeed}
Let $(T_t)_{t\in I}$ be an almost parametric binormal curvature flow on $I$ with initial datum $T_0,$ and assume that $P(T_0)\neq 0.$ There exist
a universal constant $C>0$ such that for every $t\in I$, either ${\rm supp}(T_t)$ remains at a distance at most 
$2\|T_0\|$ of ${\rm supp}(T_0)$ or
$$
{\rm supp}(T_t) \subseteq {\rm supp}(T_0) + B(V_0|t|,0),
$$
where
$$
V_0 := C \frac{\|T_0\|^3}{P(T_0)^2}.
$$
\end{prop}

\begin{cor}\label{cor:conserved}
Let $(T_t)_{t\in I}$ be an almost parametric binormal curvature flow on $I$ with initial datum $T_0,$ and assume 
$P(T_0)\neq 0.$ Then the momentum $P(T_t)$ and the angular momentum $Q(T_t)$ are independent of time. 
\end{cor}
 
\begin{rem}Up to gradient vector fields, the family $\{X_1,X_2,X_3,Y_1,Y_2,Y_3\}$ is maximal for smooth 
globally defined and linearly independent vector fields such that $D({\rm curl}(X))$ is pointwise an 
anti-symmetric matrix. In particular, there are no other ``first order'' invariants of this form. 
In contrast, smooth binormal curvature flows are known to possess
infinitely many higher order invariants (see Hasimoto \cite{Has}).   
\end{rem}

Concerning Proposition \ref{prop:finitespeed}, notice that a circle of radius $\eps>0$ gives rise to a 
traveling wave solution of \eqref{eq:strongbfbis} 
with speed $1/\eps.$ On the other hand, the current associated to such a solution (given an orientation) 
has a mass equal to $2\pi\eps$ and a momentum equal to $\pi\eps^2$. This
shows that the upper bound on the speed given by Proposition \ref{prop:finitespeed}, except for the value of $C$, is
in some sense optimal. Actually, even a curve of length of order one but small momentum may travel at
a very large speed, as the example provided by the ``bullet'' 
$\gamma_0(s):=(\frac{1}{n}\cos(ns),\frac{1}{n}\sin(ns),0)$ for $s\in \R/2\pi\Z,$ shows. In that case, the 
associated current $T_{0}$ has mass $\|T_0\|=2\pi,$ its momentum satisfies $|P(T_0)|=\frac{2\pi}{n},$ and
its speed is equal to $n.$ This suggests to raise the following: 

\begin{question}
Given a smooth solution $\gamma \: :\: \R \times \R/\ell\Z \to \R^3$ of \eqref{eq:strongbfbis} such that
the image of $\gamma(0,\cdot)$ is not entirely contained in any ball of radius $r>0.$ Is it possible to
bound its average speed (i.e. similar to the statement of Proposition \ref{prop:finitespeed}) by a function
$V_0$ depending only on $r$?  
\end{question}

\noindent{\bf Proof of Proposition \ref{prop:finitespeed}.} In view of Definition \ref{def:paraweak} and Remark
\ref{rem:ahbon} ii), it suffices to consider the case of a binormal curvature flow associated to a single smooth 
solution of \eqref{eq:strongbfbis}. Assume that ${\rm supp}(T_t)$ extends to a distance bigger than $2\|T_0\|$
from ${\rm supp}(T_0)$, fix arbitrary $a\in {\rm supp}(T_0)$ and $b\in {\rm supp}(T_t)$, and set
$$
X(x) := \chi(\|x-a\|)X_i(x-a) - \chi(\|x-b\|)X_i(x-b),
$$  
where $i\in \{1,2,3\}$ is chosen such that $|T_0(X_i)| \geq \frac{1}{\sqrt{3}}|P(T_0)|,$ and $\chi\::\:\
[0,\infty)\to [0,1]$ is a smooth cut-off function such that $\chi\equiv 1$ on $[0,\|T_0\|/2]$, $\chi \equiv 0$
outside $[0,\|T_0\|]$ and $\|\chi'\|_\infty \leq 3/\|T_0\|.$ By assumption and by construction, $X = X_i$ on 
${\rm supp}(T_0)$ and $X = -X_i$ on ${\rm supp}(T_t)$, so that 
$$
T_0(X) = P_i(T_0) \qquad \text{and} \qquad T_t(X) = -P_i(T_t) = - P_i(T_0),
$$   
where the last equality is a consequence of the conservation of momentum for smooth binormal curvature flows. On the other hand,
by Proposition \ref{prop:estimee}, we have
$$
|T_0(X) - T_t(X)| \leq C|t|^\frac12 \|T_0\| \|{\rm curl}(X)\|_\infty \leq 4C|t|^\frac12 \|T_0\|.
$$
Hence,
$$
|t| \geq \frac{P(T_0)^2}{12C^2\|T_0\|^2}.
$$
The conclusion follows by splitting the whole time interval according to subintervals on which ${\rm supp}(T_t)$ moves 
by a distance $2\|T_0\|$. \qed

\medskip
\noindent{\bf Proof of Corollary \ref{cor:conserved}.}
It suffices to use the conservation of $P$ and $Q$ at the level of the approximating smooth flows $\gamma^n$, 
to consider cut-offs of $X_1,X_2,X_3$ and $Y_1,Y_2,Y_3$ sufficiently far at infinity so that the cut-off 
does not occur on the supports of $T_t$ and $T_{\gamma^n,t}$, and to invoke pointwise in time convergence in
$\mathcal{D}'(\R^3,\R^3)$. \qed  

\subsection{Oscillations and Generalized binormal curvature flows}\label{S5.2}

The undercurrents associated to generalized binormal curvature flows, even when they can be identified
with smooth parametrized curves, need not be solutions of the classical binormal curvature flow equation
\eqref{eq:strongbfbis}. We present here a family of typical such examples, for which the speed is modified by a constant
multiplicative factor,  
and  we question about its occurrence as an almost parametrized flow according to Definition \ref{def:paraweak}.     

\medskip

\begin{prop}\label{prop:modifiedspeed}
Let  $\gamma\: : \: \R \times (\R/\ell\Z) \to \R^3$ be a smooth solution of \eqref{eq:strongbfbis}, for
some $\ell>0$, and let $(V_{\gamma,t})_{t\in\R}$ and $(T_{\gamma,t})_{t\in\R}$ denote
the associated generalized and weak binormal curvature flows, respectively, as described in
Proposition \ref{prop:consistency}.

Then for any $m > 1$ and any $a\in [ a_m, m]$, where $a_m := \frac 12 (\frac 3m - m)$, there exists a generalized
binormal curvature flow $(V^{m,a}_t)_{t\in \R}$ such that the associated undercurrents
are given by
\begin{equation}
T^{m,a}_t := T_{V^{m,a}_t} = T_{\gamma, at}
\label{eq:osc1}\end{equation}
and such that 
\begin{equation}
\mbox{ for every Borel $O\subset \R^3$, \ \ \ \ 
$V^{m,a}_t(O \times S^2) = m V_{\gamma, at}(O\times S^2)$.}
\label{extramass}\end{equation}

\end{prop}

Condition \eqref{extramass} can be thought of as asserting that these generalized solutions have ``mass $m>1$ per unit arclength". Heuristically, one may  think of the extra mass $m-1$ as corresponding to microscopic oscillations.  Note in particular that there exist generalized solutions with $a<0$ as soon as $m > \sqrt 3$.

\begin{proof}
We first show that for $m>1$ and $a\in  [ a_m, m]$, and for any
$\xi_0\in S^2$
there exists a measure $W^{m,a}[\xi_0]$ on $S^2$ such that 
\begin{equation}
\int_{S^2} \xi \ dW^{m,a}[\xi_0] = \xi_0, 
\quad\quad\quad
\int_{S^2}\xi\otimes \xi\ dW^{m,a}[\xi_0]  = a \, \xi_0\otimes \xi_0 + \frac {m-a}{3}\, Id,
\label{eq:boncurrent}\end{equation}
where $Id$ denotes the identity matrix.
Note that the second identity above implies that
\[
W^{m,a}[\xi_0] (S^2) = \int_{S^2} |\xi|^2 \ dW^{m,a}[\xi_0] = Tr \left(\int_{S^2}\xi\otimes \xi \ dW^{m,a}[\xi_0] 
\right) = m.
\]
In general 
measures $W^{m,a}[\xi_0]$ are of course not
uniquely determined by these moment conditions; the 
explicit examples we write down are chosen just for convenience.

For $\xi_0\in S^2$ and $\alpha\in (0,1],$ we define the sets
$$
S(\xi_0,\alpha) := \{\xi \in S^2\: :\: \xi\cdot \xi_0 = \alpha\},
$$
and the positive Radon measures $\mu[\xi_0,\alpha] \in \mathcal{M}(S^2)$ where
$$
\int_{S^2} f(\xi)\, d\mu[\xi_0,\alpha](\xi) := \frac{1}{\alpha}\barint_{S(\xi_0,\alpha)}
f(\xi)\,d\mathcal{H}^1(\xi) \qquad
\forall \: f\in \mathcal{C}(S^2,\R)
$$
if $\alpha < 1$ and $\mu[\xi_0,1] = \delta_{\xi_0}$ for $\alpha=1.$ 
For $\beta \ge 0$, further define 
\[
\mu[\xi_0,\alpha, \beta] = (1+\beta) \mu[\xi_0,\alpha] + \beta \delta_{-\xi_0}.
\]
One checks that for all $\alpha\in (0,1]$ and $\beta\ge 0$, 
\[
\int_{S^2} \xi \, d\mu[\xi_0,\alpha, \beta] = \xi_0 
\]
and 
\[
\int_{S^2} \xi \otimes \xi \, d\mu[\xi_0,\alpha, \beta] = \left[ (1+\beta) \frac {3\alpha^2-1}{2\alpha} + \beta \right] \xi_0\otimes \xi_0 + \frac {(1+\beta)(1-\alpha^2)}{2\alpha} Id. 
\]
Then a computation shows that 
$\mu[\xi_0,\alpha, \beta] $ satisfies the second identity in \eqref{eq:boncurrent}
if
\[
\alpha = \frac{ 2a+3+m}{3(1+m)}, \ \  \beta = \frac{m\alpha-1}{1+\alpha}.
\]
Note that since $\beta \ge 0$, we must have $\alpha \ge \frac 1m$, and clearly $\alpha\le 1$.
The requirement $\alpha\in [\frac 1m, 1]$ gives rise to 
the restriction $a\in [a_m, m]$.

Now define
\[
\int \psi(x, \xi) \,dV^{m,a}_t := \int_{\R/\ell\Z}\left(  \int_{S^2} \psi(\gamma(s), \xi) \, d W^{m,a}[\partial_s\gamma(at, s)] \right) ds.
\]
It then follows directly from \eqref{eq:boncurrent} that \eqref{extramass} is satisfied and that for every compactly supported vector field $X$,
\[
\int X \cdot\xi  \  dV^{m,a}_t  \ = \ 
\int X \cdot \xi \  dV_{\gamma, at} ,
\]
which just says that \eqref{eq:osc1} holds. In addition,
since $D({\rm curl}(X)) : Id \equiv 0$ for every $X$, we deduce from  \eqref{eq:boncurrent} and the definitions 
that
\[
\int D({\rm curl}(X)) : \xi \otimes \xi \  dV^{m,a}_t  \ = \ 
a \int D({\rm curl}(X)) : \xi \otimes \xi \  dV_{\gamma, at} .
\]
It follows from these last two identities and Proposition \ref{prop:consistency} that 
$(V^{m,a}_t)_{t\in \R}$ is a generalized binormal curvature flow.

\end{proof}

\begin{rem}
We remark that if $W^{m,a}[\xi_0]$ is any measure on $S^2$ satisfying \eqref{eq:boncurrent}, then
\[
1 = \int \xi_0 \cdot \xi \, dW^{m,a}[\xi_0] \le 
\left( \int (\xi_0 \cdot \xi)^2 \, dW^{m,a}[\xi_0] \  \int 1\  dW^{m,a}[\xi_0]\right)^{1/2}
= \sqrt{ (a + \frac{m-a}3) m}\ ,
\]
and it follows from this that $a\ge a_m$. Clearly $a\le m$,
so the restriction on the range of $a$ in \eqref{eq:boncurrent} is optimal.
In addition, if $a =  a_m$, then the above calculation implies that $\xi_0\cdot \xi$ is $W^{m,a}[\xi_0]$ a.e. constant, and from this one can check that $W^{m,a}[\xi_0]$ is supported on
$S(\xi_0,\alpha)$.
Thus the extremal case $a=a_m$
corresponds, heuristically, to microscopic oscillations whose tangents form a constant angle with the tangents of macroscopic smooth curves.
\label{rem:constantangle}\end{rem}

Varifolds with non trivial (i.e. not reduced to a single Dirac mass) dependence in $\xi$ are typically associated
to limits of wild oscillations. Indeed, the generalized binormal curvature flows described in the previous
proposition may be obtained as limits of smooth solutions of \eqref{eq:strongbfbis} (of course without the
mass convergence of the currents), at least in the case of the traveling circles
with $a= a_m$.

\begin{prop}\label{prop:cerclelent}
Let $\ell>0$ and $\gamma\: : \: \R \times (\R/\ell\Z) \to \R^3$ be a smooth solution of \eqref{eq:strongbfbis}
corresponding to a traveling circle at speed $\frac{2\pi}{\ell}.$ For every $m>1$,
there exists
a sequence $(\gamma_n)_{n\in\N}$, $\gamma_n\: :\: \R\times (\R/\ell_n\Z)\to \R^3$, of smooth solutions of
\eqref{eq:strongbfbis} such that $\ell_n \to m\ell$ and 
$$
T_{\gamma_n,t} \rightharpoonup T_{\gamma, a_m t} \qquad\text{in}\qquad \mathcal{D}'(\R^3,\R^3) 
$$ 
as $n\to +\infty,$ for all $t\in \R.$
\end{prop}
\begin{proof}
It turns out that one may actually even require the approximating solutions $\gamma_n$ to be exact traveling
wave solutions of \eqref{eq:strongbfbis}.  The latter have been extensively studied by Kida \cite{Kid} and the 
particular asymptotic required for the present proof (the $\gamma_n$ correspond to a curve with small helices wrapped around a circle) have been 
carefully detailed in \cite{JeSm2} Section 8.  In fact the
proof shows that the generalized binormal curvature flows associated to $\gamma_n$ converge to $(V^{m, a_m}_t)_{t\in \R}$, constructed in the proof of Proposition 
\ref{prop:modifiedspeed} above.
\end{proof}

\begin{question}
Given a smooth binormal curvature flow $\gamma:\: \R\times (\R/\ell \Z)\to \R^3$, and numbers $m>1$ and $a\in [a_m, m]$,  does there exist a sequence $\gamma_n\: :\: \R\times (\R/\ell_n\Z)\to \R^3$ of smooth solutions of \eqref{eq:strongbfbis} such that $\ell_n \to  m\ell $ and 
$T_{\gamma_n,t} \rightharpoonup T_{\gamma, a t}$ in the sense of distributions?
\end{question}

Even though one could expect strong instability for highly oscillatory  data, the numerics in fact tend to suggest that the answer could be 
positive, at least in the case $a=a_m$, and that corresponding choices of initial data for $\gamma_n$ would be obtained by wrapping helices of around 
the initial smooth curve $\gamma(0,\cdot)$, as is the case for the
construction in Proposition \ref{prop:cerclelent}.

\begin{rem}\label{rem:pathological}
One can use the generalized solutions of Proposition \ref{prop:modifiedspeed} to
create rather pathological examples.

For example, fix $m>1$, and let $a:\R\to [a_m, m]$ be a measurable function
that does not change sign and is a.e. bounded away from $0$.
Define $t(\tau)= \frac 1{a(\tau)} \int_0^\tau a(s) \ ds$, and let
$
V_\tau := V^{m, a(t)}_{t(\tau)}
$
for $V^{m,a}_t$ as constructed above. 
Then it is straightforward to verify
that $(V_\tau)_{\tau\in \R}$ is a generalized binormal curvature flow,
with associated undercurrents $(T_{\gamma, t(\tau)})_{\tau\in \R}$.
This illustrates quite dramatically the ill-posedness of 
the initial value problem for generalized binormal curvature flows, even if we impose
the condition that $t\mapsto V_t(\R^3\times S^2)$ is constant.

In a different direction, fix $m>1$, let $\rho: [a_m, m]\to [0,\infty)$ be a smooth function
such that $\int_{a_m}^m \rho(a) da = 1$, and define
\[
V_t = \int_{a_m}^m V^{m,a}_t  \rho(a) \ da.
\]
Then  $T_{V_0}  =  \int_{a_m}^m T_{V^{m,a}_0}  \rho(a) \ da = T_{\gamma, 0}$,
and it is easy to see that $(V_t)_{t\in \R}$ satisfies  \eqref{eq:cle} and has no boundary in the sense that 
$\int \nabla \psi \cdot \xi dV_t = 0$ for all $\psi\in C^\infty_c(\R^3)$.
But $V_t$ is not integral for times $t>0$, in the sense that the associated undercurrent
is not integral. Thus the balance law \eqref{eq:cle}  is not by itself enough to preserve integrality.
\end{rem}

\subsection{Numerical curiosities}

Our existence theory in Theorem \ref{thm:existence} allows  to consider
initial curves that have corners,  and in particular polygons. 
There are a number of open questions about the behavior of weak binormal curvature flows with polygonal initial data, many of which (uniqueness, loss of mass, ...) are special cases of more general open questions about almost parametrized weak binormal curvature flows.  
In order to possibly obtain some insight to these questions, we have performed numerical 
simulations according to an algorithm of Buttke \cite{But}, and we have observed some phenomena which we did not 
expect, which we believe are worth mentioning, and for which we have no
explanation\footnote{After all one cannot rule out a priori that the numerics are completely misleading, even if we
do not believe it is the case here.}  beyond obscure appeals to integrability (discovered for long by Hasimoto \cite{Has} for
\eqref{eq:strongbfbis}, but which is not well adapted to a non smooth setting).

\medskip
If $\gamma$ is a solution to \eqref{eq:strongbfbis}, the corresponding tangent vector $u:=\partial_s\gamma\: :\:
I\times (\R/\ell\Z)\to S^2$ satisfies 
the Schr\"odinger map equation 
\begin{equation}\label{eq:schrodimapici}
\partial_t u = u \times \partial_{ss} u.
\end{equation}
Buttke's algorithm simulates the binormal curvature flow equation \eqref{eq:strongbfbis} by the Crank-Nicolson type 
discretization
$$
\frac{u_n^{j+1}-u_n^j}{\Delta t} = \left( \frac{u_n^j + u_n^{j+1}}{2}\right) \times \left(
\frac{u_{n-1}^j+u_{n+1}^j}{2(\Delta x)^2}+\frac{u_{n-1}^{j+1}+u_{n+1}^{j+1}}{2(\Delta x)^2}  \right)
$$
of \eqref{eq:schrodimapici}, and numerical integration to recover $\gamma$ from $u.$ The implicit scheme for $u$ can be
resolved by a fixed point method if $\Delta t < \sigma (\Delta x)^2$ for some explicit $\sigma>0$; it has the advantage that the constraint
$|u_n^j|=1$, the mean $\sum_n u_n^j$, and the discrete squared $\dot H^1$ norm  $\sum_n |u_n^j-u^j_{n+1}|^2$ are conserved
quantities of the scheme.    

\medskip 

In the following pictures, we present the shape of the simulated solution at different 
(well chosen) times for a 5000 points discretization of a unit square parallel to the $xy$-plane as initial datum. 
\begin{center}
\includegraphics[width=6cm]{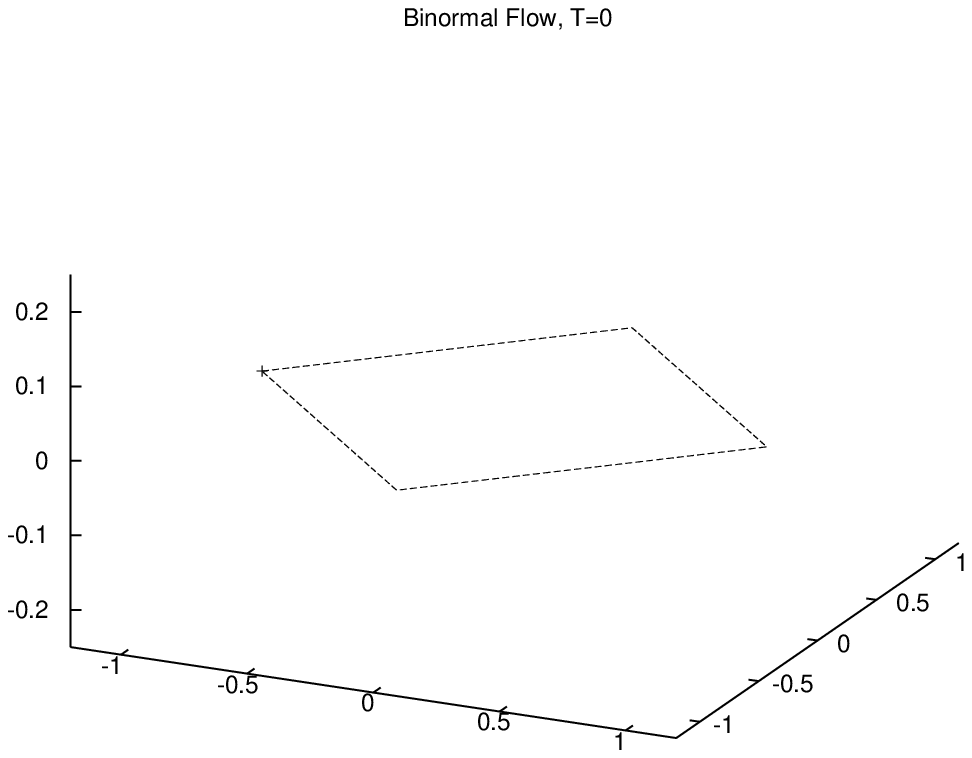} \ \includegraphics[width=6cm]{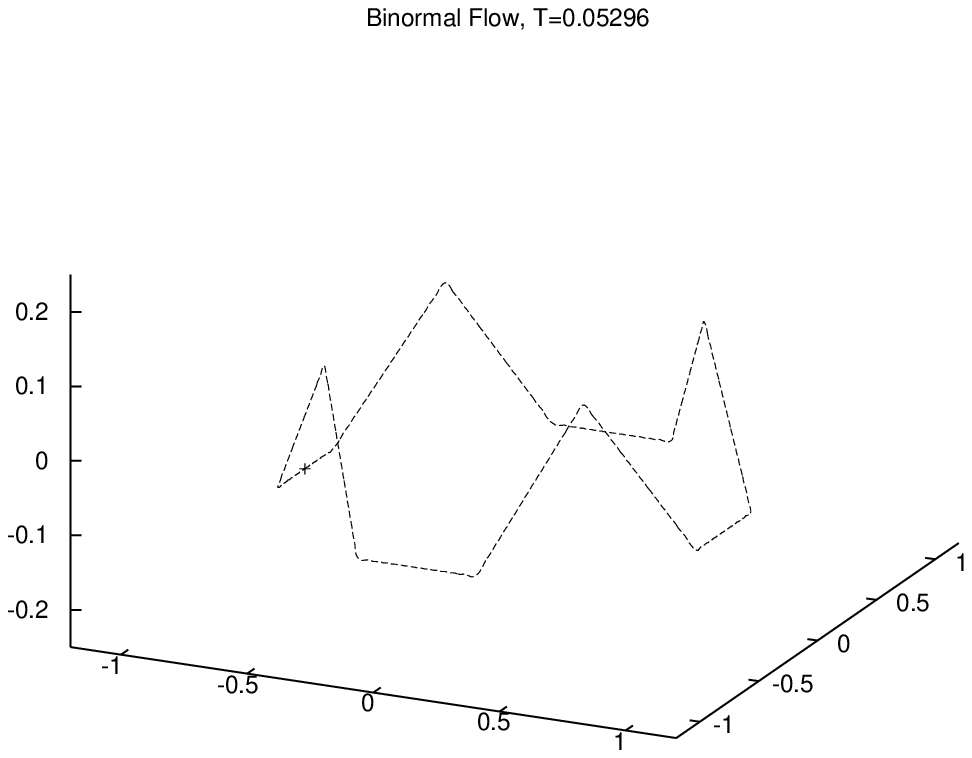}\\
\includegraphics[width=6cm]{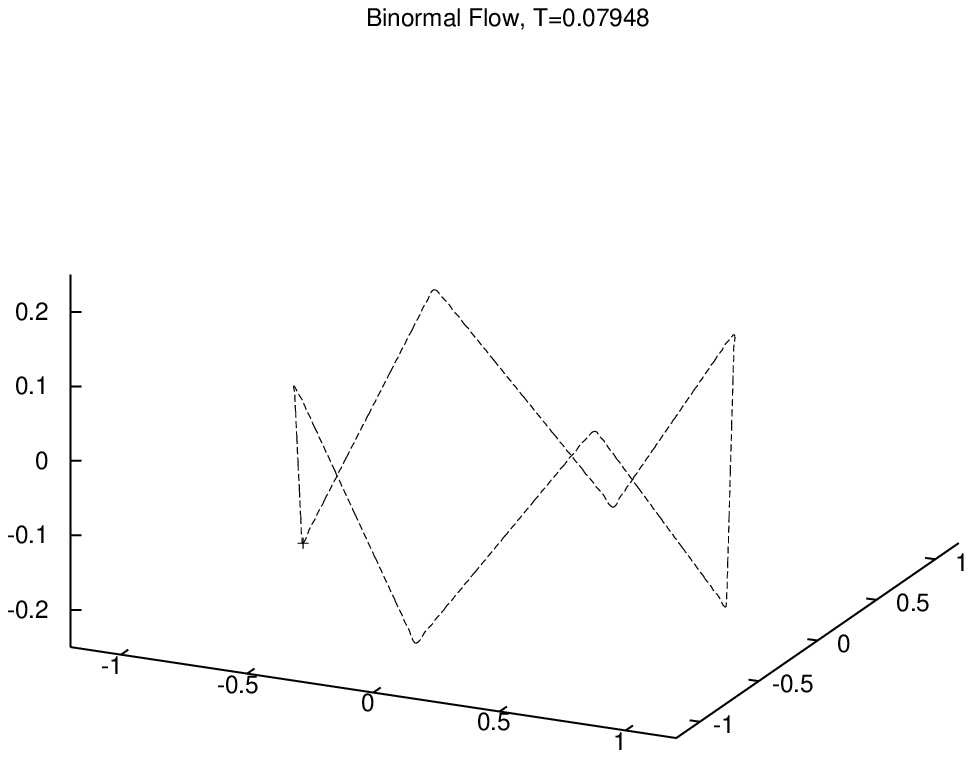} \ \includegraphics[width=6cm]{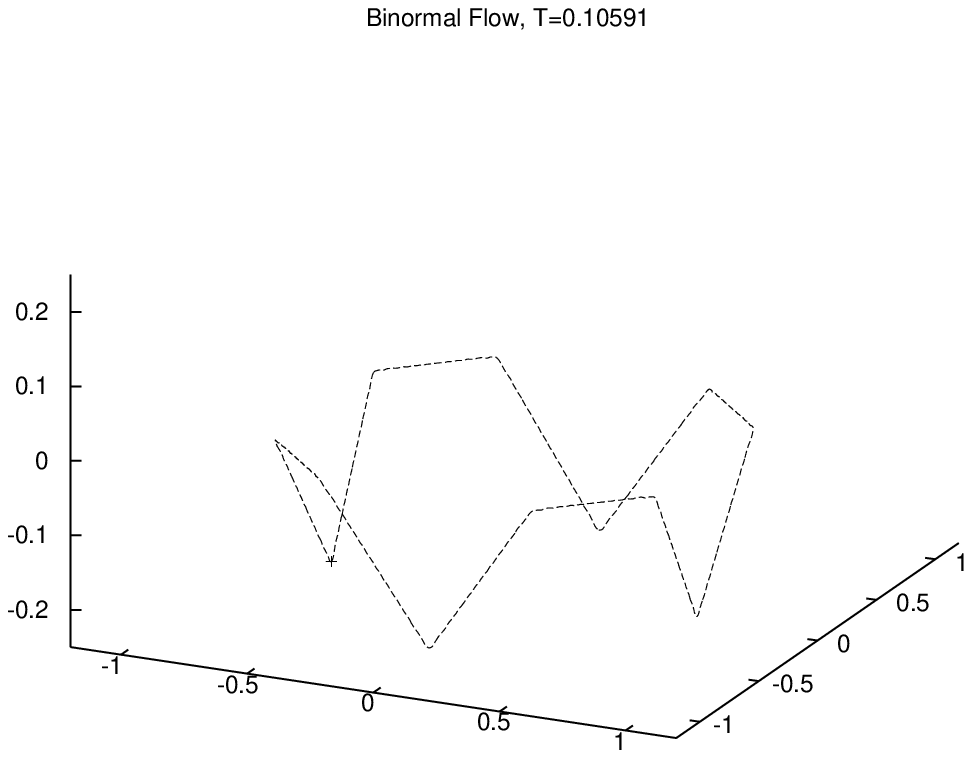}\\
\includegraphics[width=6cm]{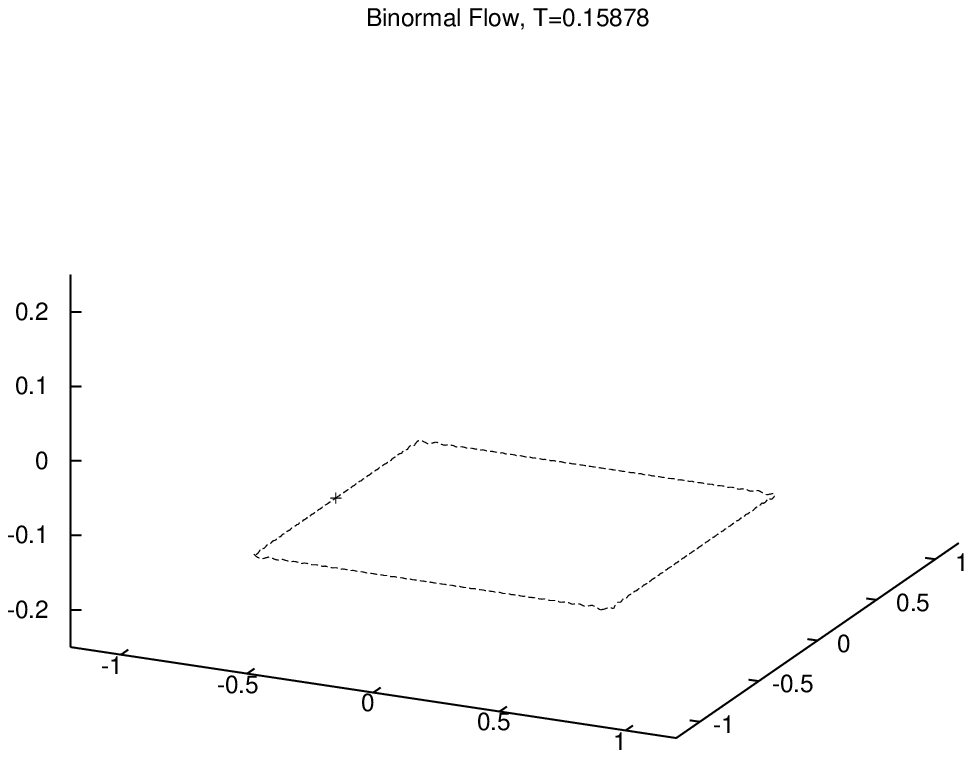}\\
\end{center}
As it may suggest, at some times 
close to 0.05296, 0.07948, 0.10591 and 0.15878, the (or ``a'') solution could become again polygonal. Notice that
the symmetries of the square are preserved (intermediate shapes have 8 or 12 sides), and that the square in the
last picture is rotated by $\pi/4$ with respect to the initial one. At times intermediate between those special
moments the simulated solution looks quite jerky and has not been represented. Also, running the simulation further in time suggests that
this sequence is reproduced in a (quasi)periodic manner. It is of course tempting to believe that solitons 
could play a role here (notice in particular the ratios of those special times); on the other hand polygons
are the worst possible examples for the Hasimoto transform (the solution is not smooth and the curvature vanishes
almost everywhere!).            

This kind of phenomena seems rather robust to some changes in the initial polygon, in particular for rectangles
or non planar initial data as the following ``half-cube'':
\begin{center}
\includegraphics[width=6cm]{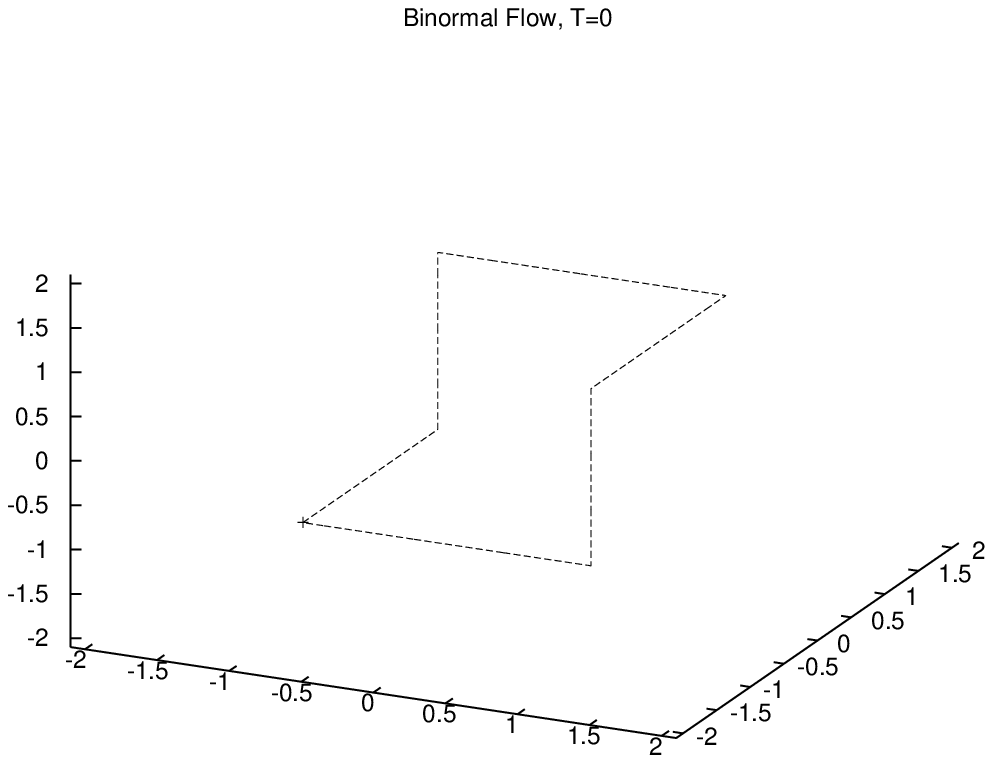}\qquad \includegraphics[width=6cm]{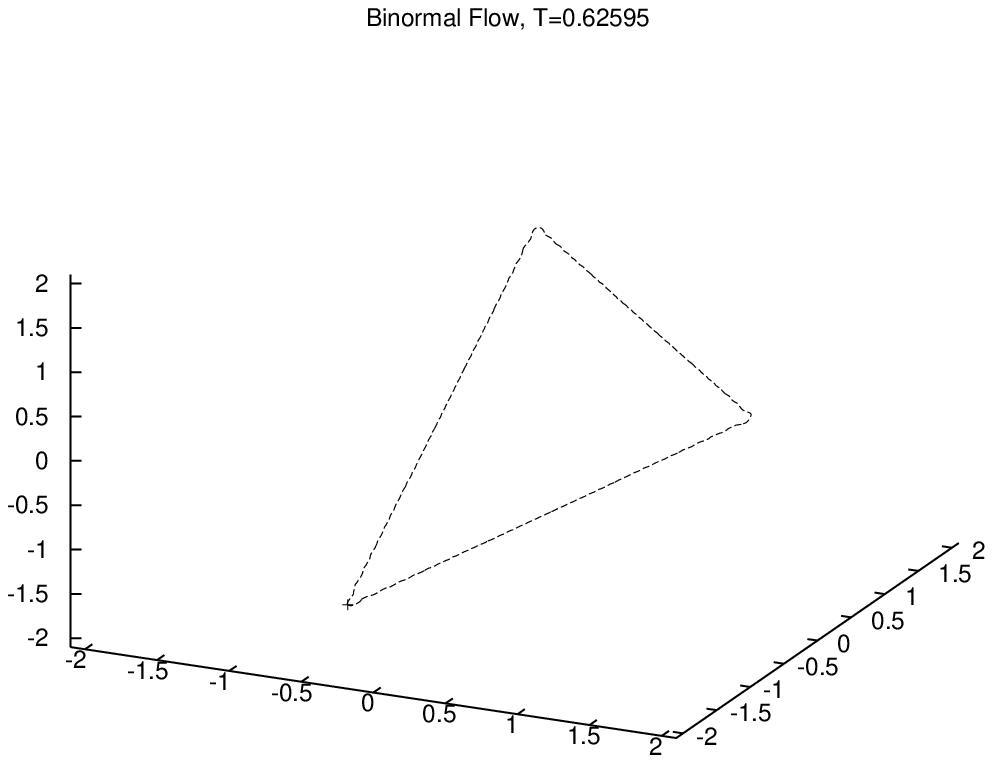}
\end{center}
(other additional times between $0$ and $0.62595$ seem to correspond to different non planar polygons (all with with 
the symmetries of the equilateral triangle), we have not included them in the picture because they are less 
distinctive on small size graphics). 

\begin{question}
Does there exist an almost parametric binormal curvature flow $(T_t)_{t\in \R}$ for which $T_t$ is the integral 1-current 
associated to an oriented polygon for at least two (and possibly an infinite sequence) of different times $t\in
\R.$ In case of positive answer, how to give an interpretation of those solutions in terms of the Hasimoto
transform and the cubic Schr\"odinger equation with Dirac masses ?     
\end{question}

\medskip

\noindent {\bf Addresses:}

\noindent
{\sc Robert L. Jerrard}: {University of Toronto, Department of Mathematics, Toronto, Ontario, M5S 2E4, Canada. E-mail: rjerrard@math.toronto.edu.

\noindent
{\sc Didier Smets}: UPMC, Universit\'e Paris 06, UMR 7598, Laboratoire Jacques-Louis Lions, BC 187, F-75252 Cedex 05, Paris, France. E-mail: smets@ann.jussieu.fr

 \end{document}